\documentclass[3p,times,final]{article}

\usepackage{amsthm,amsmath,amssymb,amsfonts}
\usepackage{multirow}
\usepackage{graphicx}
\usepackage{lineno,hyperref}
\usepackage{todonotes}
\usepackage{fullpage}

\usepackage{datetime2}

\newcommand{\R}{\mathbb{R}}\newcommand{\Rn}{{\R^n}}\newcommand{\Rnn}{{\R^{n\times n}}}
\newcommand{\N}{\mathbb{N}}
\newcommand{\C}{\mathbb{C}}
\renewcommand{\S}{\mathbb{S}}\newcommand{\Sn}{\S^n}
\renewcommand{\vec}{\mathbf}
\renewcommand{\O}{\Omega}
\DeclareMathOperator{\diam}{diam}
\renewcommand{\div}{\operatorname{div}}
\DeclareMathOperator{\diver}{div}
\DeclareMathOperator{\supp}{supp}
\DeclareMathOperator{\tri}{tri}
\DeclareMathOperator{\rect}{rect}
\DeclareMathOperator{\sinc}{sinc}
\DeclareMathOperator*{\minz}{minimize}

\theoremstyle{plain}
\newtheorem{theorem}{Theorem}
\newtheorem{lemma}[theorem]{Lemma}

\newtheorem{proposition}[theorem]{Proposition}

\theoremstyle{definition}

\newtheorem{examplex}[theorem]{Example}
\newenvironment{example}
  {\pushQED{\qed}\examplex}
  {\popQED\endexamplex}
\theoremstyle{remark}
\newtheorem{remark}[theorem]{Remark}

\begin{document}

\title{From nonlocal Eringen's model to fractional elasticity}

\author{Anton Evgrafov, Jos\'e C. Bellido}

\maketitle
%
%
%

\begin{abstract}
Eringen's model is one of the most popular theories in nonlocal elasticity.
It has been applied to many practical situations with the objective of removing
the anomalous stress concentrations around geometric shape singularities,
which appear when the local modelling is used.
Despite the great popularity of Eringen's model in mechanical engineering community,
even the most basic questions such as the existence and uniqueness of solutions have been rarely
considered in the research literature for this model.
In this work we focus on precisely these questions, proving that the model is in
general ill-posed in the case of smooth kernels, the case which appears rather
often in numerical studies.
We also consider the case of singular, non-smooth kernels, and for the
paradigmatic case of the Riesz potential we establish the well-posedness of
the model in fractional Sobolev spaces.
For such a kernel, in dimension one the model reduces to the well-known
fractional Laplacian.
Finally, we discuss possible extensions of Eringen's model to spatially heterogeneous
material distributions.

\end{abstract}






\section{Introduction}
\label{sec:intro}

Nonlocal elasticity theories have been devised with the objective of taking
into account long range internal interaction forces between particles.
In this way these theories aim at alleviating various singularity problems that
arise in the local theories, such as for example stress singularities in the
vicinity of cracks.
Beginning of the nonlocal theory of elasticity goes back to the pioneering work
of Kr{\"o}ner~\cite{kroner}.
In this early paper the classical linear local Lam{\'e} model is modified by
adding a nonlocal term in the form of an integral operator acting on the
displacements.

Perhaps the most popular and extended theory of nonlocal elasticity is the
one due to Eringen~\cite{eringen}.
In Eringen's model a nonlocal stress tensor, computed as an average of the
local stress tensor,  is introduced.
The equation of motion is then expressed in terms of the non-local stress
tensor.
If the elastic tensor is constant throughout the domain this theory can be
equivalently expressed by replacing the local strain in the constitutive
relation for the classical elasticity by a nonlocal one, obtained by
averaging~\cite{polizzotto01}.

These integral theories are called strongly nonlocal theories since
the stress at a point in the domain depends, through averaging, on the stress
at points around it.
Another class of nonlocal theories of elasticity are the so-called weakly
nonlocal theories, with the gradient theory of Aifantis~\cite{aifantis11}
arguably being the most well-known one.
In this theory the stress is expressed as a function of the strain and its
Laplacian at the same point, inducing a smoothing or regularization of strains.
This model results in a boundary value problem (BVP) associated with a fourth
order differential operator acting on the
displacements~\cite{altan-aifantis, ru-aifantis}.

In this paper we focus on integral theories, and particularly on the classical
Eringen's model of linear nonlocal elasticity.
This extremely popular model has been utilized in a variety of mechanical
applications.
Recently, it has attracted revitalized interest owing to its applicability to
the modelling of nanobeams and nanobars (see~\cite{romano} and the references
therein).
In spite of such an interest in this model from the point of view of
applications, mathematical studies of it are very scarce.
The only reference devoted to the question of existence of solutions for
the Eringen's model of linear elasticity is~\cite{altan89}.
Unfortunately the proof in the cited paper is neither complete nor correct,
as we show in Section~\ref{sec:noncoerce}.
In this work we rigorously address the question of existence of solutions to the Eringen's model.
We provide both explicit theoretical results and numerical examples demonstrating
that the model in its weak form is not necessarily coercive under the original
hypothesis of smoothness of the integral kernel, which is mathematically and
mechanically unacceptable.
A direct consequence of this fact is the highly unstable behavior of the
discretized solutions with respect to the mesh refinement.

We emphasize that on any fixed mesh the discretized non-local Eringen's model
admits solutions, which we believe explains numerous successful numerical
simulations based on this model.
Furthermore, although there are no rigorous studies to the best of authors'
knowledge, it is plausible to expect that this model converges to the local
model of linear elasticity when the long range interaction potential is scaled
appropriately (that is, when the potential converges to \(\delta\)-function in
some sense).
Rigorous limit derivations in the sense of \(\Gamma\)-convergence for other
related nonlocal models are for instance presented in~\cite{belmorped, mengesha, ponce}.
Therefore, one may expect that when the scaling is such that the non-local model
is close enough to a local model, for a given mesh size, the
discrete solutions are close to those given by the local model, while stresses
are smoother owing to the smoothing effect of the integral convolution that is
built in into the model.

In view of the ill-posedness of Eringen's model with smooth kernels it is
natural to look for possible remedies.
One such possibility is the so-called Eringen's mixture model, which has
been proposed by Eringen himself~\cite{eringen}.
This model has been recently revitalized in connection with applications in
nano-scale  modelling.
This model is, roughly speaking, a convex combination of the classical local
elasticity model and the Eringen's non-local integral model.
Another way of thinking about it is that this model can be formally obtained from the
non-local Eringen's model by adding a \(\delta\)-function to the non-local
interaction kernel.
Consequently, it should not come as a surprise that this model remains close
to the local elasticity model and inherits many theoretical properties from it.
In particular, it is elementary to see that this model is well-posed, yet for
the sake of completeness we include an existence result which is valid for
general positive definite kernels, including smooth ones.

An approach which in our opinion is much closer in spirit to the original idea
behind the non-local model of Eringen is to consider singular non-local interaction kernels.
More specifically, we will focus on the Riesz potential kernel, for which
we are able to show well-posedness, that is the existence, uniqueness, and stability
of solutions, for the nonlocal Eringen's model in fractional Sobolev spaces
\(H^s_0\), \(0\le s <1\).
Riesz potentials arise, for example, in the definition of the fractional Laplacian,
see~\cite{stein}.
Fractional Laplacian is one of the the most paradigmatic differential operators
in nonlocal modeling, with applications in many applied contexts (see the
survey~\cite{vazquez} and the references therein).
Keeping in mind that in one spatial dimension Eringen's model with Riesz potential
kernel reduces to the fractional Laplacian, we provide a very natural connection
between the non-local Eringen's model and fractional partial differential equations
in the context of linear elasticity.
This development requires new ideas and tools, such as for instance a nonlocal
version of Korn's inequality.

Finally we consider the extension of this nonlocal integral model with the Riesz
potential kernel to the case of heterogeneous materials, that is, the case of
a spatially varying stiffness tensor.
Such as an extension, being completely straightforward in the local case,
presents serious difficulties in the nonlocal situation, as far as symmetry and strict
positive definiteness of the problem are concerned.
We discuss these difficulties and propose an extension of Eringen's model with
Riesz potential to this general situation, in the sense that it coincides with
the Eringen's model for a constant stiffness tensor.
We also establish the existence and uniqueness of solutions for the heterogeneous
model.

The outline of the paper is the following: in Section~\ref{sec:nlmodel} we introduce the Eringen's
nonlocal integral model.
Section~\ref{sec:noncoerce} is devoted to the discussion of this model for a constant stiffness tensor.
First, in Subsections~\ref{subsect:noexist_l2} and~\ref{subsect:noexist_h1}, we demonstrate the ill-posedness
for smooth kernels in \(L^2\) and more general square integrable kernels in \(H^1_0\) owing to the lack of
coercivity of the bilinear form in the weak formulation of the problem.
Explicit theoretical results and numerical examples corroborating those results are given.
Subsection~\ref{subsect:exist_l2} is devoted to a simple example of a kernel for which existence of solutions holds in \(L^2\).
In~Subsection~\ref{subsec:nlKorn}, prior to proving existence of solutions for the Riesz potential kernel in Subsection~\ref{subsec:Riesz}, we prove a nonlocal Korn's inequality and coercivity and boundedness of the problem in its natural
functional space, which coincides with a fractional Sobolev space for the Riesz potential kernel.
Subsection~\ref{subsec:nlmix} is devoted to the Eringen's mixture model including a general existence result.
Finally, in Section~\ref{sec:heterog} we deal with extending the model to the heterogeneous material case.



\section{Nonlocal elasticity model}
\label{sec:nlmodel}

We consider a version of the Eringen's nonlocal elasticity model given
in~\cite{polizzotto01} (see also~\cite{eringen} and the references therein).
Let \(\O\subset\R^n\), \(n=1, 2\), or \(3\) be an open bounded domain with
Lipschitz boundary \(\Gamma = \partial \Omega\).
Let further \(\R_+ \ni d \mapsto \tilde{A}(d)\) be a function describing the
non-local interaction between the points in the model at a distance
\(d\ge 0\) from each other.
It will be convenient to evenly extend \(\tilde{A}\) onto the whole real line,
that is, we put \(\tilde{A}(d) = \tilde{A}(-d)\), \(\forall d\le 0\).
We define the kernel \(A: \R^n\times\R^n \to \R\) as
\(A(x,x') = \tilde{A}(|x-x'|)\).
We will assume that \(A \in L^1(\Omega\times\Omega)\), and that it is
a \emph{strictly} positive definite kernel:
\begin{equation}
  \label{pd}
  \int_{\O} \int_{\O} A(x,x') \phi(x) \phi(x') \,\mathrm{d}x'\,\mathrm{d}x > 0,
  \quad \forall \phi \in C^\infty_c(\Omega) \setminus \{0\},
\end{equation}
where the inequality above is known as Mercer's condition.
Often \(\tilde{A}\) is taken to be a non-negative, smooth function with
small compact support -- which results in \(A\) being a typical convolution
kernel for mollifying (see~\cite{brezis}).

The rest of this section should be understood as a preliminary informal
discussion where we do  not pay attention to the smoothness or integrability
requirements, which individual functions and the function spaces they are
contained in should satisfy.
The precise details will be added later on.

The nonlocal elasticity model that we consider can be stated as follows:
find the displacements \(u : \O \to \Rn\), the local strains
\(\varepsilon: \O \to \Sn \),
and the nonlocal stresses \(\sigma :\O \to \Sn\), where \(\Sn\) is
the set of \(n\times n\) symmetric matrices, such that
\begin{equation}\label{ne}
\left\{\begin{alignedat}{-1}
  -\diver(\sigma)&=f, && \text{in \(\O\)},\\
  \varepsilon &=\frac{1}{2}[\nabla u+(\nabla u)^\mathrm{T}], && \text{in \(\O\)},\\
  \sigma &=  \int_{\O} A(x,x')C\varepsilon(x')\,\mathrm{d}x', && \text{in \(\O\)},\\
\sigma \cdot \hat{n}&=g, && \text{on \(\Gamma_N\)},\\
u&=\bar{u}, &&\text{on \(\Gamma_D\)},\\
\end{alignedat}\right.
\end{equation}
where \(\Gamma_D,\Gamma_N \subset \Gamma\), \(\Gamma_D\cap\Gamma_N=\emptyset\),
\(\overline{\Gamma_D\cup\Gamma_N} = \Gamma\) are the Dirichlet and Neumann parts
of the boundary, respectively; \(\hat{n}\) is the outwards facing unit normal
for \(\O\) on \(\Gamma\); and \(C\) is the fourth-order stiffness tensor
with the usual symmetries.\footnote{%
Unless it is explicitly stated otherwise, we assume that the stiffness tensor
\(C\) is constant in \(\O\).}
In~\eqref{ne}, the equations from top to bottom are the equilibrium, kinematic
compatibility, and non-local constitutive equations,
and traction (Neumann) and displacement (Dirichlet) boundary conditions,
respectively.
The  boundary conditions are in turn defined by the traction forces
\(g : \Gamma_N \to \R^n\) and the prescribed displacements
\(\bar{u} : \Gamma_D \to \R^n\).
The equilibrium equations are written with respect to the applied external
volumetric forces \(f : \O \to \R^n\).

As in the case of local elasticity, we assume that the
stiffness tensor is bounded and positive definite, that is, that there exist
constants \(\overline{C}\ge \underline{C} > 0\) such that
\begin{equation}
  \label{eq:localcoercivity}
\underline{C} \epsilon: \epsilon \le
C\epsilon: \epsilon \le
\overline{C} \epsilon: \epsilon,\quad
\text{for any \(\epsilon\in \Sn\)},
\end{equation}
where \(:\) stands for the Frobenius inner product in \(\Rnn\), that is
for \(\alpha=(\alpha_{km})_{1\le k,m\le n}\) and
\(\beta=(\beta_{km})_{1\le k,m\le n}\) we put
\(\alpha:\beta=\sum_{k,m=1}^n \alpha_{km}\beta_{km}\).


The weak formulation of~\eqref{ne} is obtained in the usual manner.
Namely we substitute the kinematics and the constitutive equations into the
equilibrium equation, multiply the latter with a test function
\( v \in V = \{\, \tilde{v} : \O \to \Rn \mid \tilde{v} = 0
\text{\ on \(\Gamma_D\)} \,\}\) and integrate by parts.
As a result we obtain the problem of finding
\( u\in u_0 + V\), where \( u_0 : \O \to \Rn \) is some fixed
function satisfying the Dirichlet boundary conditions \(u_0=\bar{u}\) on
\(\Gamma_D\), and such that
\begin{equation}\label{wf}
a(u,v)=\ell(v), \quad \forall v\in V,
\end{equation}
where
\begin{equation}\label{eq:al}
  \begin{aligned}
    a(u,v)&=\int_\O\int_{ \O}
    A(x,x') C \varepsilon_u(x): \varepsilon_v(x')\,\mathrm{d}x'\,\mathrm{d}x,
    \quad \text{and}\\
    \ell(v)&=\int_\O f(x)\cdot v(x)\,\mathrm{d}x
        +\int_{\Gamma_N} g(x)\cdot v(x)\,\mathrm{d}x,
  \end{aligned}
\end{equation}
and finally
\(\varepsilon_u = [\nabla u+(\nabla u)^\mathrm{T}]/2\) and
\(\varepsilon_v = [\nabla v+(\nabla v)^\mathrm{T}]/2\).



\section{Discussion of Eringen's model}
\label{sec:noncoerce}

To the best of the authors' knowledge, the only study dedicated to the
question of existence and uniqueness of solutions to the nonlocal
Eringen's model~\eqref{ne} is~\cite{altan89}.
This study focuses on the case of homogeneous Dirichlet boundary conditions,
that is, \(\Gamma_D = \Gamma\) and \(\bar{u} = 0\).
We will now briefly recall the approach taken in~\cite{altan89}.

In view of~\eqref{pd}, the symmetric bilinear expression
\((u,v)_A=\int_\O\int_\O
A(x,x') \nabla u(x):\nabla v(x')\,\mathrm{d}x'\,\mathrm{d}x,\) defines
an inner product on \(C_c^\infty(\O; \R^n)\), and induces a norm \(\|\cdot\|_A\).
Let \(V_A\) be the completion of \(C_c^\infty(\O; \R^n)\) with respect to this
inner product.
The author of~\cite{altan89} rigorously verifies that the symmetric bilinear
form \(a(\cdot,\cdot)\) defined by~\eqref{eq:al} is bounded and coercive on
\((C_c^\infty(\O; \R^n), \|\cdot\|_A)\), and therefore also on the Hilbert space
\(V_A\).
(We note that the coercivity, in addition to~\eqref{pd}
and~\eqref{eq:localcoercivity}, relies on a non-local version of Korn's
inequality established in the cited work.\footnote{%
A general reference on existence and uniqueness of solutions in classical, local
linear elasticity, including the classical Korn's inequality
is~\cite{marsden-hughes}.})
The author of~\cite{altan89} then immediately proceeds to applying
Lax--Milgram's lemma to the problem~\eqref{wf} considered on the Hilbert space
\(V_A\), where the linear functional \(\ell\) in the right hand side is defined
by the function \(f \in L^2(\O;\Rn)\), and concludes that this problem always
posesses a unique solition \(u \in V_A\).
To expose the flaw in the argument, let us recall the Lax-Milgram's lemma
(see, for example, \cite{brezis}).

\begin{theorem}[Lax-Milgram's Lemma]\label{thm:LM}
  Let \(H\) be a Hilbert space, \(a: H\times H \to \R\) be a coercive and
  bounded bilinear form, and \(\ell : H \to \R\) be a linear bounded
  functional.
  Then there exists a unique \(u\in H\) such that
  \begin{equation*}
    a(u,v)=\ell(v), \quad \text{for all \(v\in H\)}.
  \end{equation*}
  This solution satisfies the stability estimate \(\|u\|_H \leq \alpha^{-1}\|\ell\|_{H'}\),
  where \(\alpha\) is the coercivity constant corresponding to \(a\).
  Moreover, if \(a\) is symmetric, then \(u\) is characterized as the
  unique solution of the following unconstrained optimization problem:
  \begin{equation*}
    \minz_{v \in H} I(v) := \frac{1}{2}a(v,v)-\ell(v).
  \end{equation*}
\end{theorem}

It is worth pointing out that in our case  the quadratic functional
\(I(v) := a(v,v)/2 - \ell(v)\) represents the strain energy of the
system~\cite{polizzotto01}.

At this point the reader has noticed that the last condition needed for the
successful application of Lax-Milgram's lemma, that is the boundedness of the
linear functional \(\ell\), is left unchecked, thus voiding the proof
in~\cite{altan89}.
Is this functional continuous on \(V_A\)?
The answer to this question is: it depends on the kernel \(A\).
Let us elaborate on this answer with the following discussion.

Given \(f \in L^2(\Omega)\), we would like to estimate from above the following
quantity:
\begin{equation}\label{eq:func}
  \|f\|_{V_A'} = \sup_{v \in V_A\setminus\{0\}} \frac{|f(v)|}{\|v\|_{A}}.
\end{equation}
From Cauchy-Schwartz's inequality we know that for any \(f\in L^2(\O;\Rn)\)
and any \(v \in C_c^\infty(\O;\Rn)\) we have the estimate
\(|f(v)| \le \|f\|_{L^2(\O;\Rn)} \|v\|_{L^2(\O;\Rn)}\), with equality when
\(f=\beta u\) for some \(\beta \in \R\).
Thus if \( \sup_{v \in C_c^\infty(\O;\R^n)\setminus \{0\}}
\|v\|_{L^2(\O;\Rn)}/\|v\|_{A}\) is bounded from above (that is, when
\(V_A\) is continuously embedded into \(L^2(\O;\Rn)\)), the quantity
in~\eqref{eq:func} is bounded and Theorem~\ref{thm:LM} is indeed applicable.
If, on the other hand, we can construct a sequence
\(v_k \in C_c^\infty(\O;\R^n)\setminus \{0\}\) such that
\( \lim_{k\to\infty} (v_k,v_k)_A/(v_k,v_k)_{L^2(\O;\Rn)} =0\), then
it is also quite likely that for some \(f \in L^2(\O;\Rn)\) the resulting
\(\ell\) is unbounded on \(V_A\), therefore Theorem~\ref{thm:LM} does not apply,
and in fact there may be no solutions to the problem~\eqref{wf}.

We will now demonstrate that either alternative is possible.

\subsection{Ill-posedness of~\eqref{wf} in \(L^2(\O;\Rn)\)
for smooth kernels \(\tilde{A}(|\cdot|)\)}
\label{subsect:noexist_l2}

\begin{proposition}\label{prop:compact_embedding_l2}
Assume that the function \(\R^n \ni d \mapsto \tilde{A}(|d|)\) is twice
weakly differentiable on \(\Omega-\Omega\) with
\(\Omega \times \Omega \ni (x,x') \mapsto \Delta\tilde{A}(|x-x'|)
\in L^2(\Omega \times \Omega)\).
Then \(L^2(\O;\Rn)\) is compactly embedded in \(V_A\).
\end{proposition}
\begin{proof}
Under these assumptions we can use integration by parts to rewrite the
inner product \((\cdot,\cdot)_A\) on \(C_c^\infty(\Omega;\Rn)\) as follows:
\begin{equation}
  \label{eq:laplace1}
  \begin{aligned}
    (u,v)_A &= \sum_{k,m=1}^n \int_\O\int_\O \tilde{A}(|x-x'|)
    \partial_{x_m} u_k(x) \partial_{x'_m} v_k(x')\,\mathrm{d}x'\,\mathrm{d}x
    =
    \sum_{k,m=1}^n \int_\O\int_\O \partial_{x_m}\partial_{x'_m}\tilde{A}(|x-x'|)
    u_k(x) v_k(x')\,\mathrm{d}x'\,\mathrm{d}x
    \\&=
    -\int_\O\int_\O \Delta\tilde{A}(|x-x'|)
    u(x)\cdot v(x')\,\mathrm{d}x'\,\mathrm{d}x,
  \end{aligned}
\end{equation}
where the boundary terms do not appear because \(u,v \in C_c^\infty(\O;\Rn)\).
Let us put \(K(x,x') = -\Delta\tilde{A}(|x-x'|)\).
From~\eqref{eq:laplace1} we immediately conclude that
\(\forall u \in C_c^\infty(\O;\Rn)\) we have the inequality
\(\|u\|_A^2 \le \|K\|_{L^2(\O\times\O;\R)}
\|u\|^2_{L^2(\O;\Rn)}\),
which implies continuous embedding of
\(L^2(\O;\Rn)\) into \(V_A\) owing to the density of
\(C_c^\infty(\O;\Rn)\) in both spaces.

To show that this embedding is compact it is sufficient to show that
if a sequence \(u_k \in C_c^\infty(\O;\Rn)\) converges to \(0\) weakly in
\(L^2(\O;\Rn)\) then \(\|u_k\|_A \to 0\).
However, this also follows from~\eqref{eq:laplace1} owing to the fact
that the operator \(L^2(\O;\Rn) \ni u \mapsto \int_\O K(x,x')u(x')\,\mathrm{d}x'
\in L^2(\O;\Rn)\) is compact, hence also a completely continuous operator
(as any Hilbert--Schmidt integral operator) \cite{brezis}.
\end{proof}

Proposition~\ref{prop:compact_embedding_l2} implies that
\(V_A\) cannot be continuously embedded into \(L^2(\O;\Rn)\).\footnote{%
Indeed, if the embedding was continuous, the identity operator
\(i: L^2(\O;\Rn)\to L^2(\O;\Rn)\) would be compact as a composition
\(i = i_{V_A \to L^2(\O;\Rn)} \circ i_{L^2(\O;\Rn)\to V_A}\) of
a compact and a continuous embedding operators, which is impossible in
infinite-dimensional spaces.}
In view of the previous discussion, in this case we cannot guarantee that
the linear functional in the right hand side of~\eqref{wf} is bounded and
therefore also existence of solutions to~\eqref{wf}.
More generally, equation~\eqref{eq:laplace1} shows the equivalence
between~\eqref{wf} and a Fredholm integral equation of the first
kind with kernel \(K(x,x') = -\Delta\tilde{A}(|x-x'|)\).
Fredholm equation of the first kind is a canonical ill-posed problem~\cite{pog}.

Let us illustrate the situation with the following one-dimensional example.

\begin{figure}
  \centering
  \includegraphics[width=0.45\columnwidth]{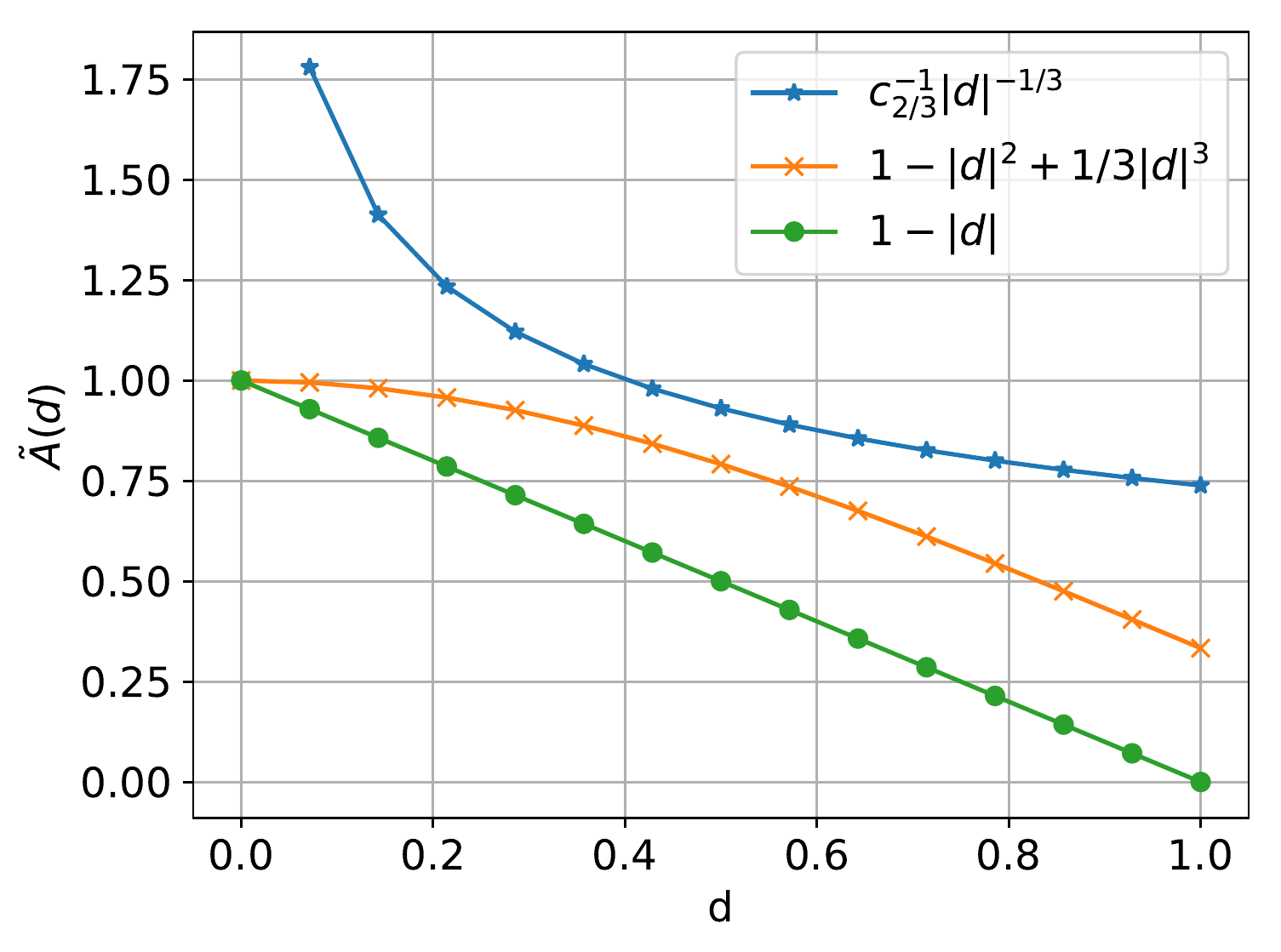}
  \caption{Kernel-generating functions \(\tilde{A}\),
  featuring in the examples.}
  \label{fig:1}
\end{figure}

\begin{example}\label{ex:noexist}
  Let \(n=1\), \(\O=(0,1)\), \(C=1\), \(f = 1\),
  and \(\tilde{A}(d) = 1-d^2 + \tfrac{1}{3} d^3\),
  see Figure~\ref{fig:1}.
  Then \(-\Delta \tilde{A}(|d|) = 2(1-|d|)
  = 2\tri(d)\) on \(\O-\O = (-1,1)\),
  where \(\tri\) is the \emph{triangle function}.

  We will utilize the Fourier transform
  \(\mathcal{F}\{\phi\}(\xi) =
  \int_\R \phi(x)\exp(-2\pi i x\xi)\,\mathrm{d}x\),
  where we will implicitly extend all functions by \(0\) outside of
  their domain of definition.
  Since \(\mathcal{F}: L^2(\R) \to L^2(\R;\C)\) is unitary, for any
  \(u,v \in C_c^\infty(\Omega)\) we can write
  \begin{equation*}
    \begin{aligned}
      (u,v)_A &=
      \int_\R \int_\R \tilde{A}(|x-x'|) \nabla u(x) \cdot \nabla v(x')
      \,\mathrm{d}x'\,\mathrm{d}x
      =
      (\nabla u,\tilde{A}(|\cdot|) * \nabla v)_{L^2(\R;\R)}
      =
      (\mathcal{F}\{\nabla u\},\mathcal{F}\{\tilde{A}(|\cdot|) * \nabla v\})_{L^2(\R;\C)}
      \\ &=
      (2\pi i \xi \mathcal{F}\{u\},
       2\pi i \xi \mathcal{F}\{\tilde{A}(|\cdot|)\}\mathcal{F}\{v\})_{L^2(\R;\C)}
      =
       (\mathcal{F}\{u\},
        4\pi^2|\xi|^2 \mathcal{F}\{\tilde{A}(|\cdot|)\}\mathcal{F}\{v\})_{L^2(\R;\C)}
      \\&=
        (\mathcal{F}\{u\},
         \mathcal{F}\{-\Delta \tilde{A}(|\cdot|)\}\mathcal{F}\{v\})_{L^2(\R;\C)}
      =
         2(\mathcal{F}\{u\},
          \sinc^2(\xi)\mathcal{F}\{v\})_{L^2(\R;\C)},
    \end{aligned}
  \end{equation*}
  where \(\sinc(x) = \sin(\pi x)/(\pi x)\), and we have used the standard
  properties of Fourier transform (convolution theorem, transform of the
  derivatives).
  Note that the derivation above can be used as an alternative way of
  arriving at~\eqref{eq:laplace1}.
  This simple calculation shows that integral operator with the kernel
  \(-\Delta \tilde{A}(|x-x'|)\) is an almost everywhere positive \emph{Fourier
  multiplier} \(2\sinc^2(\xi)\),
  and as such is strictly positive definite.
  Consequently \((\cdot,\cdot)_A\) is indeed an inner product on
  \(C_c^\infty(\Omega)\).

  We can do a similar calculation with the right hand side of~\eqref{wf}:
  \begin{equation*}
    \begin{aligned}
      (f,v)_{L^2(\O;\R)} &=
      (\mathcal{F}\{\rect(\cdot-0.5)\},\mathcal{F}\{v\})_{L^2(\R;\C)}
      =
      (\exp(-\pi i \xi)\sinc(\xi), \mathcal{F}\{v\})_{L^2(\R;\C)},
    \end{aligned}
  \end{equation*}
  where \(\rect(\cdot)\) is the characteristic function of the interval
  \((-0.5,0.5)\).
  As a result, the solution \(u\) should be equal to
  \(\tfrac{\pi}{2} \mathcal{F}^{-1}\{\exp(-\pi i \xi) \xi/\sin(\pi\xi)\}\).
  Clearly, the function in the curly brackets is not in \(L^2(\R;\C)\),
  and therefore the Eringen problem with this kernel
  does not admit a solution in \(L^2(\O)\).

  Let us now perform a conforming finite element simulation of~\eqref{wf}
  with this kernel.
  We subdivide \(\O\) into \(N\) uniform subintervals
  \(I_k\), \(k=1,\dots,N\) of length \(h=1/N\) and put
  \begin{equation*}
    V_{A,h,p} = \{\, \phi \in C^0(\O) \mid \phi(0)=\phi(1)=0, \phi|_{I_k}
    \text{\ is a polynomial of degree \(\le p\)}, k=1,\dots,N\,\}.%
    \footnote{%
    Note that \(V_{A,h,p} \subset H^1_0(\O)\) is a subspace of \(V_a\),
    see Proposition~\ref{prop:compact_embedding_h1}.%
    }
  \end{equation*}
  Let \(e_k\), \(k=1,\dots,\tilde{N}(N,p)\) be a basis in \(V_{A,h,p}\).
  We compute the matrices \(K\) and \(M\) with elements
  \(K_{k,m} = (e_k,e_m)_A=a(e_k,e_m)\) and \(M_{k,m}=(e_k,e_m)_{L^2(\O)}\).
  The smallest eigenvalue \(\lambda_{h,p}\) corresponding to the generalized
  eigenvalue problem \(K \vec{v} = \lambda M \vec{v}\) admits the variational
  characterization
  \begin{equation}\label{eq:lambda_hp}
    \lambda_{h,p} = \inf_{v_h \in V_{A,h,p}\setminus\{0\}}
    \frac{(v_h,v_h)_A}{(v_h,v_h)_{L^2(\O;\R)}}.
  \end{equation}
  The behaviour of this eigenvalue for a range of \(h=N^{-1}\) and \(p=1,2\)
  is shown in Figure~\ref{fig:2}~(a).
  This figure illustrates two already established facts:
  1. For each \(h>0\) the matrix \(K\) is symmetric and positive definite,
  since \((\cdot,\cdot)_A\) is, and therefore
  the discretization of~\eqref{wf} admits a unique solution for an arbitrary
  \(f\in L^2(\Omega)\);
  2. For small \(h\), \(\lambda_{h,p} = O(h^2)\) and therefore ``in the limit''
  the ratio
  \(\inf_{v \in V_{A}\setminus\{0\}} {(v,v)_A}/{(v,v)_{L^2(\O)}}=0\),
  in accordance with the compact embedding established in
  Proposition~\ref{prop:compact_embedding_l2}.

  \begin{figure}
    \centering
    \begin{tabular}{cc}
    \includegraphics[width=0.45\columnwidth]{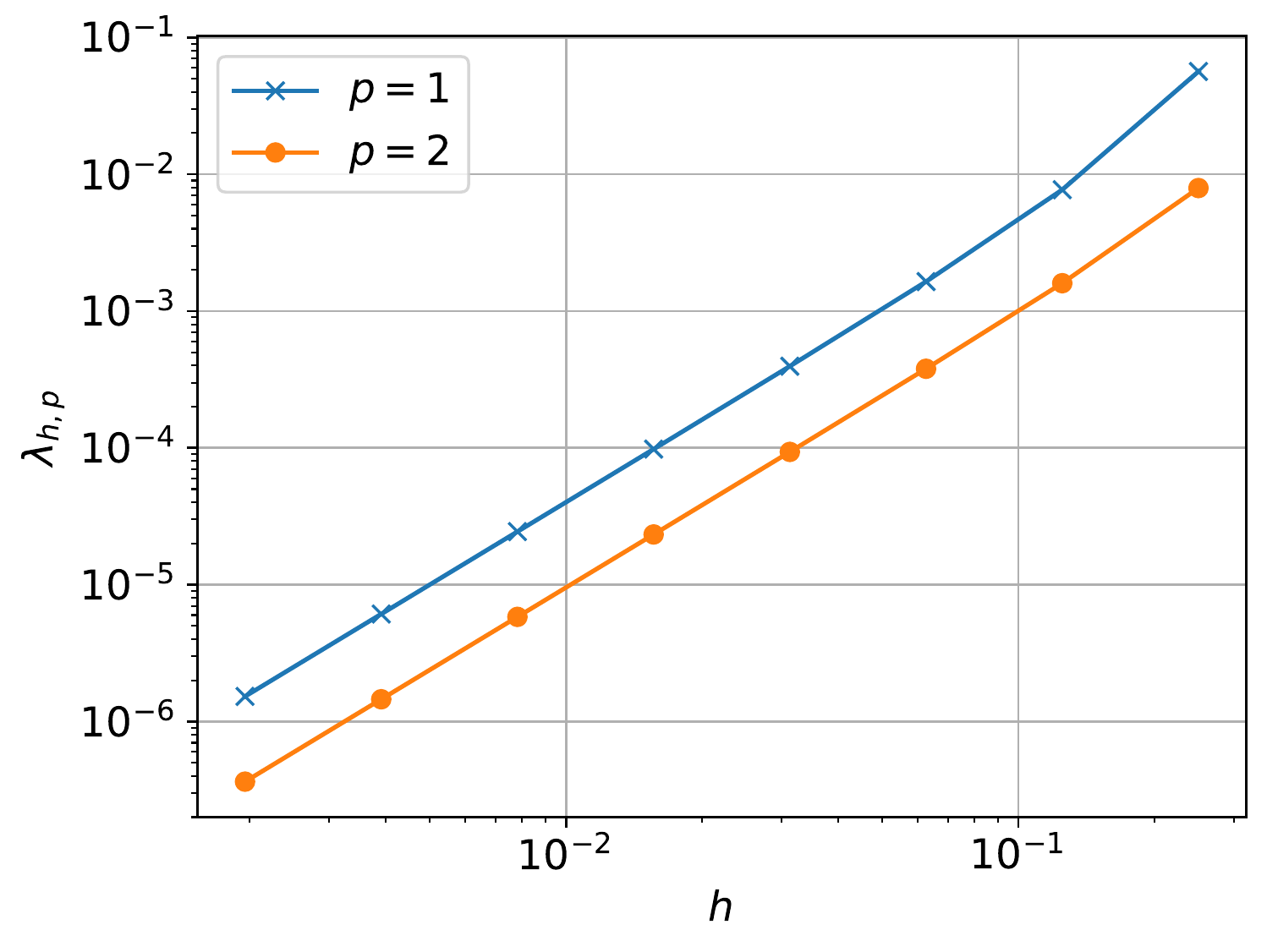} &
    \includegraphics[width=0.45\columnwidth]{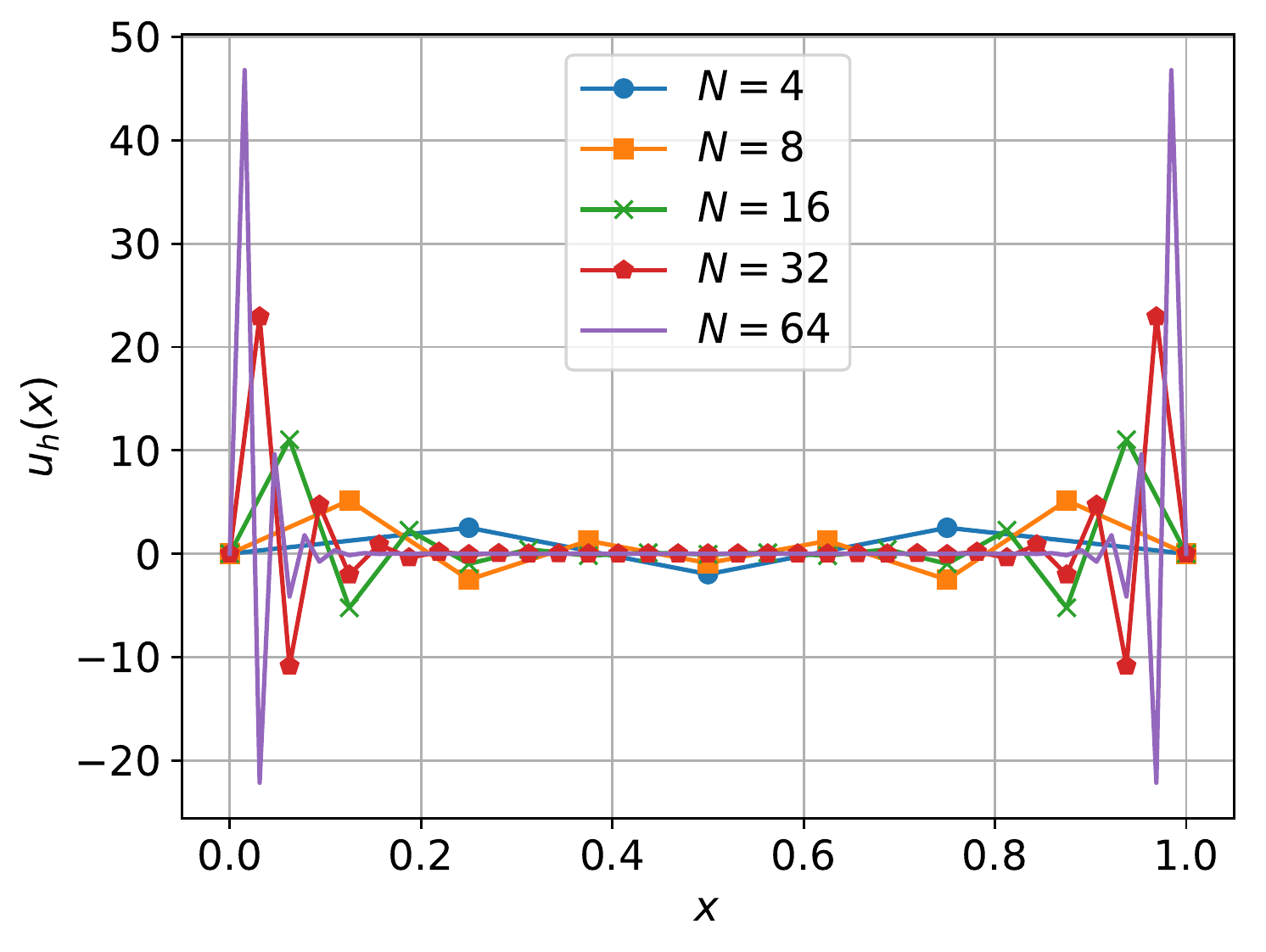} \\
    (a) & (b)
    \end{tabular}
    \caption{%
    (a): Behaviour of the generalized eigenvalue \(\lambda_{h,p}\),
         see~\eqref{eq:lambda_hp};
    (b): Solutions to the discretized system~\eqref{wf}, obtained
         on a uniform grid with \(N\) elements, see Example~\ref{ex:noexist}.
         In this figure, we show the results obtained using linear finite
         elements; quadratic elements lead to similar results.%
    }
    \label{fig:2}
  \end{figure}

  Finally, solving the discretized problems corresponding to \(f = 1\)
  on a sequence of refined meshes we obtain progressively more oscillatory
  discrete solutions shown in Figure~\ref{fig:2}~(b), which is in accordance with
  the non-existence of solutions asserted earlier.
\end{example}

\subsection{Ill-posedness of~\eqref{wf} in \(H^1_0(\O;\Rn)\)}
\label{subsect:noexist_h1}

Before we proceed to ``positive'' existence results, we would like to
eliminate another possibility for solvability of~\eqref{wf}.

\begin{proposition}\label{prop:compact_embedding_h1}
  Under the assumption \(A \in L^2(\O\times\O)\),
   \(H^1_0(\O;\Rn)\) is compactly embedded into \(V_A\).
\end{proposition}
\begin{proof}
  The embedding is continuous because \(\forall u,v \in C_c^\infty(\O;\Rn)\)
  we have the bound
  \begin{equation*}
    \begin{aligned}
      |(u,v)_A| &\le \|A\|_{L^2(\O\times\O)} \|\nabla u\|_{L^2(\O;\Rnn)}
      \|\nabla v\|_{L^2(\O;\Rnn)} \le
      \|A\|_{L^2(\O\times\O)} \|u\|_{H^1_0(\O;\Rn)}
      \|v\|_{H^1_0(\O;\Rn)},
    \end{aligned}
  \end{equation*}
  and because \(C_c^\infty(\O;\Rn)\) is dense in \(V_A\) and \(H^1_0(\O;\Rn)\).
  For compactness of the embedding the same arguments as in the proof of
  Proposition~\ref{prop:compact_embedding_l2} apply, but to the integral
  operator with kernel \(A\) while taking into account that the weak
  convergence in \(H^1_0(\O;\Rn)\) implies the weak convergence of the gradients
  in \(L^2(\O;\Rnn)\).
\end{proof}
As a consequence, \(V_A\) cannot be continuously embedded into \(H^1_0(\O;\Rn)\),
and therefore we must have
\begin{equation*}
  \inf_{v \in C_c^\infty(\O;\R^n)\setminus \{0\}}
  \frac{(v,v)_A}{\|v\|^2_{H^1(\O;\Rn)}} = 0,
\end{equation*}
and the bilinear form \(a(\cdot,\cdot)\) is not coercive on \(H^1_0(\O;\Rn)\).
Theorem~\ref{thm:LM} is not applicable to~\eqref{wf} when considered over
\(H^1_0(\O;\Rn)\).

\subsection{Existence of solutions in \(L^2(\O;\Rn)\)}
\label{subsect:exist_l2}

%

By dropping the assumptions of Proposition~\ref{prop:compact_embedding_l2}
we can recover the existence of solutions to~\eqref{wf} in, for example,
\(L^2(\O;\Rn)\).
Consider the following example.

\begin{example}\label{ex:exist_l2}
  Let \(C=1\), \(\O=(0,1)\), \(\tilde{A}(d) = 1-d\).
  Notice that \(\tilde A(|\cdot|)\) is not differentiable at the origin.
  Then for any \(v \in C_c^\infty(\Omega;\R)\) we have
  \begin{equation*}
    \begin{aligned}
      &\|v\|_A^2 =
      \int_0^1 \int_0^1 (1-|x-y|)v'(x)v'(y)\,\mathrm{d}y\,\mathrm{d}x
      =
      \bigg(\underbrace{\int_0^1 v'(x)\,\mathrm{d}x}_{=0}\bigg)^2
      -2\int_0^1 \int_0^x (x-y) v'(x)v'(y)\,\mathrm{d}y\,\mathrm{d}x
      \\ &=
      -2\int_0^1 x v'(x)\int_0^x v'(y)\,\mathrm{d}y\,\mathrm{d}x
      +2\int_0^1 v'(x)\int_0^x yv'(y)\,\mathrm{d}y\,\mathrm{d}x
      \\&=
      -2\int_0^1 x v'(x) v(x)\,\mathrm{d}x
      +2\int_0^1 v'(x)\int_0^x [(yv(y))'-v(y)]\,\mathrm{d}y\,\mathrm{d}x
      =
      -2\int_0^1 v'(x)\int_0^x v(y)\,\mathrm{d}y\,\mathrm{d}x
      \\&=
      -2\int_0^1 v(y)\int_y^1 v'(x)\,\mathrm{d}x\,\mathrm{d}y
      = 2\|v\|^2_{L^2(\O;\R)}.
    \end{aligned}
  \end{equation*}
  Owing to the polarization identity this in fact implies that
  \(\forall u,v \in C_c^\infty(\O;\R): (u,v)_A = 2(u,v)_{L^2(\O;\R)}\).\footnote{%
  An even easier way of seeing this is by noting
  that \(-\Delta \widetilde{A}(|\cdot|) = 2\delta(\cdot)\) (or \(2\delta(\cdot)-\delta(\cdot-1)
  -\delta(\cdot+1)\) if \(\widetilde{A}(|\cdot|)\) is extended by zero outside of
  \(\O-\O\)), in the sense of distributions.
  This can then be used in~\eqref{eq:laplace1} immediately resulting in
  \((u,v)_A = (u,-\Delta \widetilde{A}(|\cdot|)*v)_{L^2(\R)} = 2(u,v)_{L^2(\R)}
  = 2(u,v)_{L^2(\O)}\)
   (or, respectively, in
  \((u,v)_A = (u,-\Delta \widetilde{A}(|\cdot|)*v)_{L^2(\R)} = 2(u,v)_{L^2(\R)}
  -(u,v(\cdot-1))_{L^2(\R)} -(u,v(\cdot+1))_{L^2(\R)}
  = 2(u,v)_{L^2(\O)}\)).}
  Since \(C_c^\infty(\O;\R)\) is dense in both \(V_A\) (per definition) and
  \(L^2(\O;\R)\) (owing to mollifying), we have that the spaces
  \(V_A\) and \(2^{-1/2} L^2(\O;\R)\) are isometrically isomorphic.
  The functional \(\ell\) defined by~\eqref{eq:al} is bounded on \(V_A\),
  and Theorem~\ref{thm:LM} is applicable.

  Clearly the problem~\eqref{wf} is uniquely solvable for any
  \(f\in L^2(\O)\) by \(u = \frac{1}{2} f\).
  In particular the solution possesses no smoothness beyond that of \(f\),
  nor does it satisfy the homogeneous Dirichlet boundary conditions,
  unless \(f\) does.
\end{example}

\subsection{Non-local Korn's inequality, and coercivity and boundedness in \(V_A\)}
\label{subsec:nlKorn}

Before we go any further with positive existence results in specific spaces for
certain kernels, we would like to establish the fact that under relatively mild
assumptions the bilinear form \(a(\cdot,\cdot)\) is coercive and bounded on
\(V_A\), that is, it satisfies the assumptions of Lax--Milgram's lemma.
Such results, including a non-local version of Korn's inequality, have been
proved in~\cite{altan89} under strong assumptions on the kernel \(A\),
including its continuity, which are too restrictive for our purposes.

We begin with the case of square-integrable kernels \(A \in L^2(\O\times\O)\).
Note that for such kernels, owing to the density of \(C^\infty_c(\O)\)
in \(L^2(\O)\), the non-strict version of inequality~\eqref{pd} holds for all
\(\phi \in L^2(\O)\).

The proofs of boundedness and coercivity of \(a(\cdot,\cdot)\) in the case
of square integrable kernels
rely on the spectral theorem for compact self-adjoint operators, which is
restated here for the reader's convenience; for details see for
example~\cite[Chapter~6]{brezis}.

\begin{theorem}[Spectral decomposition of compact self-adjoint operators]
  \label{thm:spectral}
  Let \(\O\subset \Rn\) be an open bounded domain, and let
  \(K \in L^2(\O\times\O)\) be s symmetric kernel.
  Then there is an infinite sequence of real eigenvalues
  \(|\lambda_1| \geq |\lambda_2| \geq \dots \geq 0\),
  \(\lim_{k\to\infty} \lambda_k=0\),
  and a corresponding sequence of eigenfunctions \(\varphi_k\), \(k=1,2,\dots\),
  which forms an orthonormal basis in \(L^2(\O)\),
  solving the eigenvalue problem
  \begin{equation*}
    \int_\O K(x,x')\varphi(x')\,\mathrm{d}x' = \lambda \varphi(x).
  \end{equation*}
\end{theorem}

Note that under the additional assumption that the kernel \(K(\cdot,\cdot)\)
in the previous theorem is strictly positive definite in the sense of~\eqref{pd},
all eigenvalues \(\lambda_k\) must be non-negative,  as
\(\lambda_k = \lambda_k(\varphi_k,\varphi_k)_{L^2(\O)} =
\int_\O K(x,x')\varphi(x')\varphi_k(x)\,\mathrm{d}x'\,\mathrm{d}x \ge 0\),
see the comment before the statement of the theorem.

\begin{proposition}\label{prop:VAbound}
  Suppose that \(A \in L^2(\O\times\O)\) is a strictly positive kernel and
  the stiffness tensor \(C\) satisfies~\eqref{eq:localcoercivity}.
  Then for all \(u,v \in V_A\) we have the inequality
  \(|a(u,v)|\leq \overline{C} \|u\|_A \|v\|_A\).
\end{proposition}
\begin{proof}
  Owing to the density of \(C_c^\infty(\O;\Rn)\) in \(V_A\) it is sufficient
  to prove this inequality for \(u,v \in C_c^\infty(\O;\Rn)\).

  Let \(C^{1/2}\) be a square root of the positive definite symmetric tensor
  \(C\).
  Owing to the strict positive definiteness of \(A\) the bilinear form
  \(a(\cdot,\cdot)\) is an inner product on \(C_c^\infty(\O;\Rn)\), since for
  each \(u = C_c^\infty(\O;\Rn)\)
  we have the non-negativity \(a(u,u)=\int_\O \int_\O A(x,x') (C^{1/2} \varepsilon_u(x)):
  (C^{1/2}\varepsilon_u(x'))\,\mathrm{d}x\,\mathrm{d}x'\geq 0\)
  with equality only when
  \(C^{1/2} \varepsilon_u=0\), and thereby owing to the (local) Korn's inequality
  also \(u=0\).
  Therefore, Cauchy--Bunyakovsky--Schwarz inequality applies to \(a(\cdot,\cdot)\)
  and for each  \(u,v \in C_c^\infty(\O;\Rn)\) we can write
  \(|a(u,v)|^2 \leq a(u,u) a(v,v)\).
  Thus to prove the claim it is sufficient to show that
  \(a(u,u) \leq \overline{C} (u,u)_A\), \(\forall u \in C_c^\infty(\O;\Rn)\).

  Let \((\lambda_k,\varphi_k)\) be the eigenvalue-eigenfunction pairs
  for the integral operator with kernel \(A\) provided by Theorem~\ref{thm:spectral}.
  We expand \(\nabla u\) in the basis \(\varphi_k\), that is,
  we put  \(G_{u,k} = \int_\O \varphi_k(x)\nabla u(x)\,\mathrm{d}x \in\Rnn\).
  The corresponding expansion coefficients for strains
  are clearly \(\varepsilon_{u,k} = (G_{u,k}+G_{u,k}^{\mathrm{T}})/2 \in\S^n\).

  As a result we have the string of inequalities
  \begin{equation*}
  \begin{aligned}
  a(u,u)&=\int_\O\int_\O A(x,x') \varepsilon_u(x):\varepsilon_u(x')\,\mathrm{d}x'\,\mathrm{d}x
  =\sum_{k,m=1}^\infty C\varepsilon_{u,k}:\varepsilon_{u,m}
  \int_\O\int_\O A(x,x')\varphi_k(x)\varphi_m(x')\,\mathrm{d}x'\,\mathrm{d}x \\
  &=\sum_{k=1}^\infty \lambda_k C\varepsilon_{u,k}:\varepsilon_{u,k}
  \leq \overline{C} \sum_{k=1}^\infty \lambda_k \varepsilon_{u,k}:\varepsilon_{u,k}
  \leq \overline{C} \sum_{k=1}^\infty \lambda_k G_{u,k}:G_{u,k} \\
  &=\overline{C} \sum_{k=1}^\infty
  \int_\O\int_\O A(x,x')G_{u,k}:G_{u,k} \varphi_k(x)\varphi_k(x')\,\mathrm{d}x'\,\mathrm{d}x
  =\overline{C} \sum_{k,m=1}^\infty
  \int_\O\int_\O A(x,x') G_{u,k}:G_{u,m}\varphi_k(x)\varphi_m(x')\,\mathrm{d}x'\,\mathrm{d}x \\
  &=\overline{C}\int_\O\int_\O A(x,x') \nabla u(x):\nabla u(x')\,\mathrm{d}x'\,\mathrm{d}x
  = \overline{C} (u,u)_A,
  \end{aligned}
  \end{equation*}
  where we have utilized the non-negativity of eigenvalues \(\lambda_k\),
  \(L^2(\O)\)-orthonormality of the eigenfunctions \(\varphi_k\),
  and the orthogonality of symmetric and skew-symmetric tensors yeilding the
  inequality \(\varepsilon_{u,k}:\varepsilon_{u,k} \leq
  G_{u,k}:G_{u,k}\).
\end{proof}

The coercivity of \(a(\cdot,\cdot)\) relies on the following non-local version
of Korn's inequality, which is stated for \(n\geq 2\) since the result for
\(n=1\) is trivial.

\begin{lemma}[Non-local Korn's inequality]\label{korn}
  For \(n\ge 2\), let \(A\in L^2(\O\times \O)\) be a symmetric positive definite kernel such
  that its (weak) partial derivatives are in \(L^1_{loc}(\O\times \O)\).
  Additionally, we assume that these derivatives verify the equality
  \begin{equation}\label{eq:hypderiv}
    \partial_{x_k} A(x,x')=\gamma \partial_{x_k'}A(x,x'),\quad \text{a.e.\ }(x,x')\in \O\times\O,
  \end{equation}
  where \(\gamma^2 = 1\).
  Then for all \(u\in V_A\) we have the inequality
  \begin{equation*}
    2\int_\O\int_\O A(x,x') \varepsilon_u(x):\varepsilon_u(x')\,\mathrm{d}x'\,\mathrm{d}x\ge (u,u)_A.
  \end{equation*}
\end{lemma}
\begin{proof}
  Note that in view of Proposition~\ref{prop:VAbound} (applied with \(C=I\),
  the identity tensor) and the density of \(C_c^\infty(\O;\Rn)\) in \(V_A\),
  it is sufficient to establish the claim for \(u\in C_c^\infty(\O;\Rn)\).
  For any such smooth function with compact support we have
 \begin{equation*}
 \begin{aligned}
 \int_\O\int_\O A(x,x') \varepsilon_u(x)&:\varepsilon_u(x')\,\mathrm{d}x'\,\mathrm{d}x
 =\frac{1}{4} \sum_{k,m=1}^n \int_\O\int_\O A(x,x')\left[ \left( \partial_{x_m} u_k(x)+\partial_{x_k}u_m(x)\right)
 \cdot\left( \partial_{x_m'} u_k(x')+\partial_{x_k'}u_m(x')\right)\right]\,\mathrm{d}x'\,\mathrm{d}x \\
 &=\frac{1}{2} \sum_{k,m=1}^n \int_\O\int_\O A(x,x') \left[ \left(\partial_{x_k}u_m(x)\partial_{x_k'}u_m(x')\right)
  +\left( \partial_{x_k}u_m(x)\partial_{x_m'}u_k(x')\right)\right] \,\mathrm{d}x'\, \mathrm{d}x\\
 &=\frac{1}{2}\int_\O\int_\O A(x,x') \nabla u(x):\nabla u(x')\,\mathrm{d}x'\,\mathrm{d}x
 +\frac{1}{2}\int_\O\int_\O A(x,x') \left( \partial_{x_k}u_m(x)\partial_{x_m'}u_k(x')\right) \,\mathrm{d}x'\, \mathrm{d}x
\end{aligned}
\end{equation*}
 Applying now successively the definition of weak derivative and the hypothesis~\eqref{eq:hypderiv},
 we get the string of equalities
 \begin{equation*}
 \begin{aligned}
 &\int_\O\int_\O A(x,x') \partial_{x_k } u_m(x) \partial_{x_m'} u_k(x')\,\mathrm{d}x'\,\mathrm{d}x
 = -\int_\O\int_\O \partial_{x_k}A(x,x') u_m(x) \partial_{x_m'} u_k(x')\,\mathrm{d}x'\,\mathrm{d}x
 \\&= -\gamma\int_\O\int_\O  \partial_{x_k'}A(x,x') u_m(x) \partial_{x_m'} u_k(x')\,\mathrm{d}x'\,\mathrm{d}x
 = \gamma\int_\O\int_\O A(x,x') u_m(x) \partial_{x_k'x_m'}^2 u_k(x')\,\mathrm{d}x'\,\mathrm{d}x \\
 &=-\gamma\int_\O\int_\O  \partial_{x_m'}A(x,x') u_m(x) \partial_{x_k'} u_k(x')\,\mathrm{d}x'\,\mathrm{d}x
 =-\gamma^2\int_\O\int_\O  \partial_{x_m}A(x,x') u_m(x) \partial_{x_k'} u_k(x')\,\mathrm{d}x'\,\mathrm{d}x \\
 &=\int_\O\int_\O  A(x,x') \partial_{x_m}u_m(x) \partial_{x_k'} u_k(x')\,\mathrm{d}x'\,\mathrm{d}x.
 \end{aligned}
 \end{equation*}
 Consequently, we can write
 \begin{equation*}
 \begin{aligned}
 & \int_\O\int_\O A(x,x') \varepsilon_u(x):\varepsilon_u(x')\,\mathrm{d}x'\,\mathrm{d}x
 \\&=\frac{1}{2}\int_\O\int_\O A(x,x') \nabla u(x):\nabla u(x')\,\mathrm{d}x'\,\mathrm{d}x
  + \frac{1}{2}\int_\O\int_\O A(x,x')\div_x (u(x))\div_{x'}(u(x'))\,\mathrm{d}x'\,\mathrm{d}x \\
 & \ge \frac{1}{2} \int_\O\int_\O A(x,x') \nabla u(x):\nabla u(x')\,\mathrm{d}x'\,\mathrm{d}x
  =\frac{1}{2}(u,u)_A,
 \end{aligned}
 \end{equation*}
 where we have utilized the positive definiteness of the kernel \(A\) and
 the definition of the inner product \((\cdot,\cdot)_A\).
\end{proof}

\begin{proposition}\label{prop:VAcoercive}
  Suppose that \(A\) is a strictly positive definite kernel, which verifies
  the assumptions of Lemma~\ref{korn}, and
  the stiffness tensor \(C\) satisfies~\eqref{eq:localcoercivity}.
  Then for all \(u \in V_A\) we have the inequality
  \(a(u,u)\geq \tfrac{1}{2}\underline{C} (u,u)_A\).
\end{proposition}
\begin{proof}
  Note that in view of Proposition~\ref{prop:VAbound} and the density of
  \(C_c^\infty(\O;\Rn)\) in \(V_A\), it is sufficient to establish the claim
  for \(u\in C_c^\infty(\O;\Rn)\).

  The proof follows the lines of that of Proposition~\ref{prop:VAbound} and
  is only included to keep this document self-contained.
  Let \((\lambda_k,\varphi_k)\) be the eigenvalue-eigenfunction pairs
  for the integral operator with kernel \(A\) provided by
  Theorem~\ref{thm:spectral}.
  We expand the strains \(\varepsilon_u = (\nabla u + (\nabla u)^{\mathrm{T}})/2\)
  in the basis \(\varphi_k\), that is,
  we put  \(\varepsilon_{u,k} = \int_\O \varphi_k(x)\varepsilon_u(x)\,\mathrm{d}x
  \in \S^n\).

  Then we have the string of inequalities
  \begin{equation*}
  \begin{aligned}
  a(u,u)&=\int_\O\int_\O A(x,x') C\varepsilon_u(x):\varepsilon_u(x')\,\mathrm{d}x'\,\mathrm{d}x
  =\sum_{k,m=1}^\infty C\varepsilon_{u,k}:\varepsilon_{u,m}
  \int_\O\int_\O A(x,x')\varphi_k(x)\varphi_m(x')\,\mathrm{d}x'\,\mathrm{d}x\\
  &=\sum_{k=1}^\infty \lambda_k C\varepsilon_{u,k}:\varepsilon_{u,k}
  \geq \underline{C} \sum_{k=1}^\infty \lambda_k \varepsilon_{u,k}:\varepsilon_{u,k}
  =\underline{C} \sum_{k=1}^\infty
  \int_\O\int_\O A(x,x')\varepsilon_{u,k}:\varepsilon_{u,k} \varphi_k(x)\varphi_k(x')\,\mathrm{d}x'\,\mathrm{d}x
  \\&=\underline{C} \sum_{k,m=1}^\infty
  \int_\O\int_\O A(x,x') \varepsilon_{u,k}:\varepsilon_{u,m}\varphi_k(x)\varphi_m(x')\,\mathrm{d}x'\,\mathrm{d}x
  =\underline{C}\int_\O\int_\O A(x,x') \varepsilon_u(x):\varepsilon_u(x')\,\mathrm{d}x'\,\mathrm{d}x
  \ge \frac{1}{2}\underline{C} (u,u)_A,
  \end{aligned}
  \end{equation*}
  where we have utilized the non-negativity of eigenvalues \(\lambda_k\),
  \(L^2(\O)\)-orthonormality of the eigenfunctions \(\phi_k\),
  and Lemma~\ref{korn}.
\end{proof}

We are now ready to drop the square integrability and differentiability requirements,
which are not satisfied by the singular kernels we want to utilize in what follows.

\begin{theorem}\label{thm:L1_bnd_coerc}
  Assume that \(A(\cdot,\cdot)\in L^1(\O\times \O)\) is symmetric and strictly
  positive definite, cf.~\eqref{pd}.
  Then the conclusions of Propositions~\ref{prop:VAbound} and~\ref{prop:VAcoercive}
  hold for this kernel.
  That is, the bilinear form \(a(\cdot,\cdot)\) defined in~\eqref{eq:al} is bounded
  and coercive in \(V_A\).
\end{theorem}
\begin{proof}
  Note that both inequalities established in Propositions~\ref{prop:VAbound}
  and~\ref{prop:VAcoercive} are continuous with respect to the kernel \(A\).
  Namely, let us consider an arbitrary but fixed \(u,v \in C^\infty_c(\O;\Rn)\),
  and let \(\tilde{\O}\) be an open set containing \(\supp(u) \cup \supp(v)\),
  which is itself contained inside some compact set \(K \subset \O\),
  that is, \(\tilde{\O}\subset\subset\O\).
  We will construct a sequence of square integrable
  kernels~\(A_k \in L^2(\tilde{\O}\times\tilde{\O})\), each satisfying the assumptions
  of Propositions~\ref{prop:VAbound} and~\ref{prop:VAcoercive}.
  Furthermore, this sequence will converge strongly in \(L^1(\tilde{\O}\times\tilde{\O})\) towards \(A\).
  By considering the limits of the inequalities established in Propositions~\ref{prop:VAbound}
  and~\ref{prop:VAcoercive} we obtain the inequalities
  \begin{equation*}
    \frac{1}{2} \underline{C} (u,u)_A \le a(u,u),
    \quad\text{and}\quad
    a(u,v) \le \overline{C} \|u\|_A\|v\|_A,
  \end{equation*}
  and thereby prove the claim owing to the density of \(C^\infty_c(\O;\Rn)\) in \(V_A\).

  The announced sequence of kernels will be obtained by mollifying \(A\) with strictly positive definite
  kernels satisfying certain smoothness and symmetry requirements, which are necessary for applying
  the non-local Korn's inequality, see Proposition~\ref{korn}.

  Let \(\rho: \R \to\R_+\) be a compactly supported positive function, such that
  the resulting convolution kernel \(\Rn\times\Rn \ni (x,x')\mapsto \rho(|x-x'|) \in \R\)
  is a strictly positive definite mollifying kernel of class \(C^1_c(\Rn\times\Rn)\).
  For specific examples of such functions see for example~\cite{wendland1995a,buhmann2001a}.
  For a sequence of \(\epsilon_k \to 0\) we consider a sequence of mollified kernels
  \begin{equation*}
    A_k(x,x')=\int_\O \int_\O \rho_{\epsilon_k}(|x-z|)A(z,z')\rho_{\epsilon_k}(|x'-z'|)\,\mathrm{d}z\,\mathrm{d}z',
  \end{equation*}
  where \(\rho_{\epsilon}(d) = \epsilon^{-n}\rho(d/\epsilon)\).
  Owing to the construction, \(A_k \to A\), strongly in \(L^1(\O\times \O)\).
  Furthermore, each kernel \(A_k\) is symmetric and is in \(C^1(\Rn\times\Rn)\)
  with compact support; in particular it is in \(L^2(\O\times\O)\).

  The strict positive definiteness of \(A_k\) in \(\tilde{\O}\),
  follows from that of \(A\) and \(\rho\) as follows.
  For an arbitrary \(\phi \in C^\infty_c(\tilde{\O})\) let us put
  \begin{equation*}
    \phi_{\epsilon_k}(z)
    = \int_\O \rho_{\epsilon_k}(|z-x|)\phi(x)\,\mathrm{d}x
    = \int_\Rn \rho_{\epsilon_k}(|z-x|)\phi(x)\,\mathrm{d}x
    = (\rho_{\epsilon_k}*\phi)(z) \in C^\infty_c(\Rn),
  \end{equation*}
  where as before we extend \(\phi\) by zero outside \(\O\), and the inclusion
  is owing to the fact that both functions have compact support and
  \(\phi\) is smooth.
  For a fixed \(K\) determined by \(\tilde{\O}\) and not \(\phi\),
  it is always possible to select \(k_0\in\N\) large enough such that
  for all \(k\geq k_0\) we have the inclusion
  \(\supp(\phi_{\epsilon_k}) \subset\subset \Omega\), independently from
  \(\phi \in C^\infty_c(\tilde{\O})\).
  Consequently \(\phi_{\epsilon_k} \in C^\infty_c(\O)\).
  Finally, for all \(k\geq k_0\) we can write
  \begin{equation}\label{eq:pd_l1}
    \int_\O\int_\O A_k(x,x')\phi(x)\phi(x')\,\mathrm{d}x'\,\mathrm{d}x
    =
    \int_\O\int_\O A(z,z')
    \phi_{\epsilon_k}(z)
    \phi_{\epsilon_k}(z')\,\mathrm{d}z'\,\mathrm{d}z
    \ge 0,
  \end{equation}
  with equality only when \(\phi_{\epsilon_k}=0\) in \(\O\),
  owing to the strict positive definiteness of \(A\) in \(\O\).
  The equality \(\phi_{\epsilon_k}=0\) in turn leads to the equality
  \begin{equation*}
    0 =
    \int_\O \phi(z) \phi_{\epsilon_k}(z)\,\mathrm{d}z
    =\int_\Rn \int_\Rn \rho_{\epsilon_k}(|x-z|)\phi(x)\phi(z)\,\mathrm{d}x\,\mathrm{d}z,
  \end{equation*}
  which in view of strict positive definiteness of \(\rho_{\epsilon_k}\) in \(\Rn\)
  implies that \(\phi = 0\) in \(\Rn\).
  Therefore, the strict positive definiteness of \(A_k\) in \(\tilde{\O}\) is
  established.

  Finally, it remains to verify the differentiability of \(A_k\) and conditions
  on the derivatives, which are needed for the application of Proposition~\ref{korn}.
  The differentiability follows from that of \(\rho\), and directly from the
  construction of \(A_k\) and the symmetry of \(A\) we get the desired condition
  \(\partial_{x_k}A = \partial_{x'_k}A\).
\end{proof}

\subsection{Riesz potential and existence of solutions in \(H^{s}_0(\O;\Rn)\)}
\label{subsec:Riesz}

Whereas Example~\ref{ex:exist_l2} is mathematically very satisfying, in terms of
mechanical modelling it is much less so.
Indeed, fulfilling the prescribed Dirichlet boundary conditions may be viewed
as a very basic requirement for a mathematical model of an elastic body.
The fundamental issue is that the function space, in which the existence of
solutions has been established, does not allow a definition of \emph{trace},
which mathematically encapsulates the concept of boundary conditions.
On the other hand, we cannot expect solutions in the ``very regular'' space
\(H^1_0(\O;\Rn)\), as has been discussed in Subsection \ref{subsect:noexist_h1}.
However, we can still obtain solutions in some intermediate spaces between
\(L^2(\O;\Rn)\) and \(H^1_0(\O;\Rn)\), namely the fractional Sobolev spaces
\(H^s_0(\O;\Rn)\), \(0 < s < 1 \), see~\cite{di2012hitchhikers}.

We recall that for \(0<s<1\) the fractional Sobolev space
\begin{equation*}
  H^s(\O;\Rn) = \{\, v \in L^2(\O;\Rn) \mid
  \frac{|v(x)-v(x')|^2}{|x-x'|^{n+2s}} \in L^2(\O\times\O;\Rn)\,\},
\end{equation*}
is a Hilbert space with respect to the inner product
\((u,v)_{H^s(\O;\Rn)} = (u,v)_{L^2(\O;\Rn)} + [u,v]_{H^s(\O;\Rn)}\),
where the symmetric positive semi-definite bilinear form
\begin{equation*}
  [u,v]_{H^s(\O;\Rn)} = \int_\O \int_\O
  \frac{(u(x)-u(x'))\cdot(v(x)-v(x'))}{|x-x'|^{n+2s}}
  \,\mathrm{d}x'\,\mathrm{d}x,
\end{equation*}
generates the Gagliardo seminorm on \(H^s(\O;\Rn)\) via
\(v\mapsto [v,v]_{H^s(\O;\Rn)}^{1/2}\).
Finally, \(H^s_0(\O;\Rn)\) is defined as the closure of \(C^\infty_c(\O;\Rn)\)
in \(H^s(\O;\Rn)\).

For \(0< \alpha < n\) we put the kernel-defining function
\(\tilde{A}_{\alpha}(d) = c_\alpha^{-1} d^{\alpha-n}\), where the normalization
constant \(c_\alpha = \pi^{n/2}2^\alpha \Gamma(\alpha/2)/\Gamma((n-\alpha)/2)\).
This interaction kernel defines Riesz potential~\cite{landkof1972foundations}.
The \(n\)-dimensional Fourier transform of this function is
\(\mathcal{F}\{\tilde{A}_{\alpha}(|\cdot|)\}=|2\pi\xi|^{-\alpha}\), thereby
intimately linking it to the fractional Laplace operator~\cite{di2012hitchhikers}.

\begin{proposition}\label{prop:isomorph_Hs0}
  For \(\alpha \in ((0,2]\cap (0,n)) \setminus \{1\}\),
  \(V_A=H^s_0(\O;\Rn)\) with \(s = 1-\alpha/2 \in [0,1)\setminus\{1/2\}\),
  in the sense that the two spaces are continuously embedded into each other.
\end{proposition}
\begin{proof}
  For any \(u \in C^\infty_c(\Omega)\), and \(\alpha\), \(s\) as defined above
  we have
  \begin{equation}\label{eq:FourierRiesz}
    \begin{aligned}
      (u,u)_A &=
      \int_\Rn |2\pi\xi|^{-\alpha} |\mathcal{F}\{\nabla u\}(\xi)|^2\,\mathrm{d}\xi
      =
      \int_\Rn |2\pi\xi|^{2-\alpha} |\mathcal{F}\{u\}(\xi)|^2\,\mathrm{d}\xi
      = \frac{c(n,s)}{2}[u,u]_{H^s(\Rn;\Rn)},
    \end{aligned}
  \end{equation}
  where the first equality is owing to~\cite[Lemma~2, pp. 117]{stein}, and
  the last equality is owing to~\cite[Proposition~3.4]{di2012hitchhikers}
  with \(c(n,s)^{-1}=\int_{\Rn} (1-\cos(\zeta_1))|\zeta|^{-2s-n}\,\mathrm{d}\zeta\).
  (As before, we implicitly extend by \(0\) the functions outside their domain
  of definition).

  Fractional Friedrichs's inequality~\cite[Corollary~3.3.6]{webb2012analysis}
  (see also~\cite{ervin2006variational,ervin2007variational}) yields:
  \begin{equation*}
    \|u\|^2_{H^s(\Rn;\Rn)} \le
    (1+\diam(\O)^{2s}(\Gamma(1+s))^{-2})[u,u]_{H^s(\Rn;\Rn)},
  \end{equation*}
  and ultimately  \(\|u\|^2_{H^s(\O;\Rn)}\le\|u\|^2_{H^s(\Rn;\Rn)}\).
  Therefore \(V_A\) is continuously embedded into \(H^s_0(\O;\Rn)\).

  On the other hand, the right hand side of~\eqref{eq:FourierRiesz} is majorized
  by \(\tfrac{1}{2}c(n,s)\|u\|^2_{H^s(\Rn;\Rn)}\), and therefore
  the closure of \(C^\infty_c(\O;\Rn)\) in \(H^s(\Rn;\Rn)\), often denoted
  by \(\tilde{H}^s(\O;\Rn)\), is continuously embedded into \(V_A\).
  According to~\cite[Theorem~3.33]{mclean2000strongly},
  \(\tilde{H}^s(\O;\Rn)\) and \(H^s_0(\O;\Rn)\) are isomorphic provided that
  \(s\neq 1/2\), the case which we specifically exclude.
\end{proof}

\begin{remark}
  Whenever \(s>1/2\) in
  Proposition~\ref{prop:isomorph_Hs0}, that is, when \(0<\alpha < 1\),
  functions in \(V_A\) have a well defined concept of trace,
  see~\cite[Theorem~3.38]{mclean2000strongly}.
  In fact, if \(0\le s < 1/2\) then \(V_A = H^s(\O;\Rn)\),
  whereas for \(1/2 < s \le 1\) it holds that
  \(V_A = \{\, v \in H^s(\O;\Rn)\mid v|_{\partial\O} = 0\,\}\),
  see~\cite[Theorem~3.40]{mclean2000strongly}.
\end{remark}

After this preparatory work we are ready to state the existence result.

\begin{theorem}\label{thm:exist_Hs0}
  For \(\alpha\) and \(s\) as in Proposition~\ref{prop:isomorph_Hs0}
  the problem~\eqref{wf} with homogeneous Dirichlet boundary conditions
  (that is, \(\Gamma_D=\partial\O\) and \(\bar u=0\)) admits a unique solution
  in \(H^s_0(\O;\Rn)\).
\end{theorem}
\begin{proof}
  We apply Theorem~\ref{thm:LM}.
  Boundedness of \(\ell(\cdot) = (f,\cdot)_{L^2(\O;\Rn)}\) on
  \(V_A = H^s_0(\O;\Rn)\) (cf.~Proposition~\ref{prop:isomorph_Hs0}) for
  any \(f \in L^2(\O;\Rn)\) is a consequence of the
  fact that \(H^s_0(\O;\Rn)\) is continuously embedded into \(L^2(\O;\Rn)\),
  which follows immediately from the trivial inequality
  \(\|\cdot\|_{L^2(\O;\Rn)} \le \|\cdot\|_{H^s(\O;\Rn)}\),
  see the definition of the fractional Sobolev norm.
  The coercivity and the boundedness of \(a(\cdot,\cdot)\) on \(V_A\)
  are established in Theorem~\ref{thm:L1_bnd_coerc}.
\end{proof}

\begin{example}\label{ex:exist_Hs0}
  Consider the problem~\eqref{wf} with \(n=1\), \(\O=(0,1)\), \(C=1\),
  and \(\alpha=2/3\).
  The resulting \(\tilde{A}\) is shown in Figure~\ref{fig:1}.

  Proposition~\ref{prop:isomorph_Hs0} implies that the bilinear form
  \(a(\cdot,\cdot)\) is coercive and continuous with respect to
  the inner product on \(H^{s}_0(\O)\) with \(s = 1-\alpha/2 = 2/3\).
  We would like to verify numerically, that this is indeed the case.
  To this end, we use a finite element discretization with piecewise-linear
  elements on a uniform grid
  to assemble the non-local stiffness matrix \(K\) as was done in
  Example~\ref{ex:noexist}.
  For the current kernel we need to evaluate the double integrals of the
  Riesz kernel against constants (derivatives of the piecewise-linear basis
  functions inside each element), which can be easily done analytically.
  Instead of assembling the non-local mass matrix directly from the definition
  of the inner product in \(H^{s}_0(\O)\), we use an equivalent simpler
  construction described in~\cite{arioli2009discrete}.
  Namely, we assemble two Gramm matrices \(M_0\) and \(M_1\) corresponding
  to \(L^2(\O)\) and \(H^1_0(\O)\) inner products, respectively,
  and then put \(M=M_0 (M_0^{-1}M_1)^{s}\).
  We then compute the smallest and the largest eigenvalues
  \(\lambda_{\min}\) and \(\lambda_{\max}\) corresponding to the
  generalized eigenvalue problem \(K\boldsymbol{v} = \lambda M\boldsymbol{v}\),
  characterizing respectively the coercivity and the boundedness of
  \(a(\cdot,\cdot)\) with respect to
  the inner product on \(H^{s}_0(\O)\); see Example~\ref{ex:noexist}.
  The behaviour of the resulting eigenvalues as a function of the
  element size is shown in Figure~\ref{fig:4}.

  \begin{figure}
    \centering
    \includegraphics[width=0.45\columnwidth]{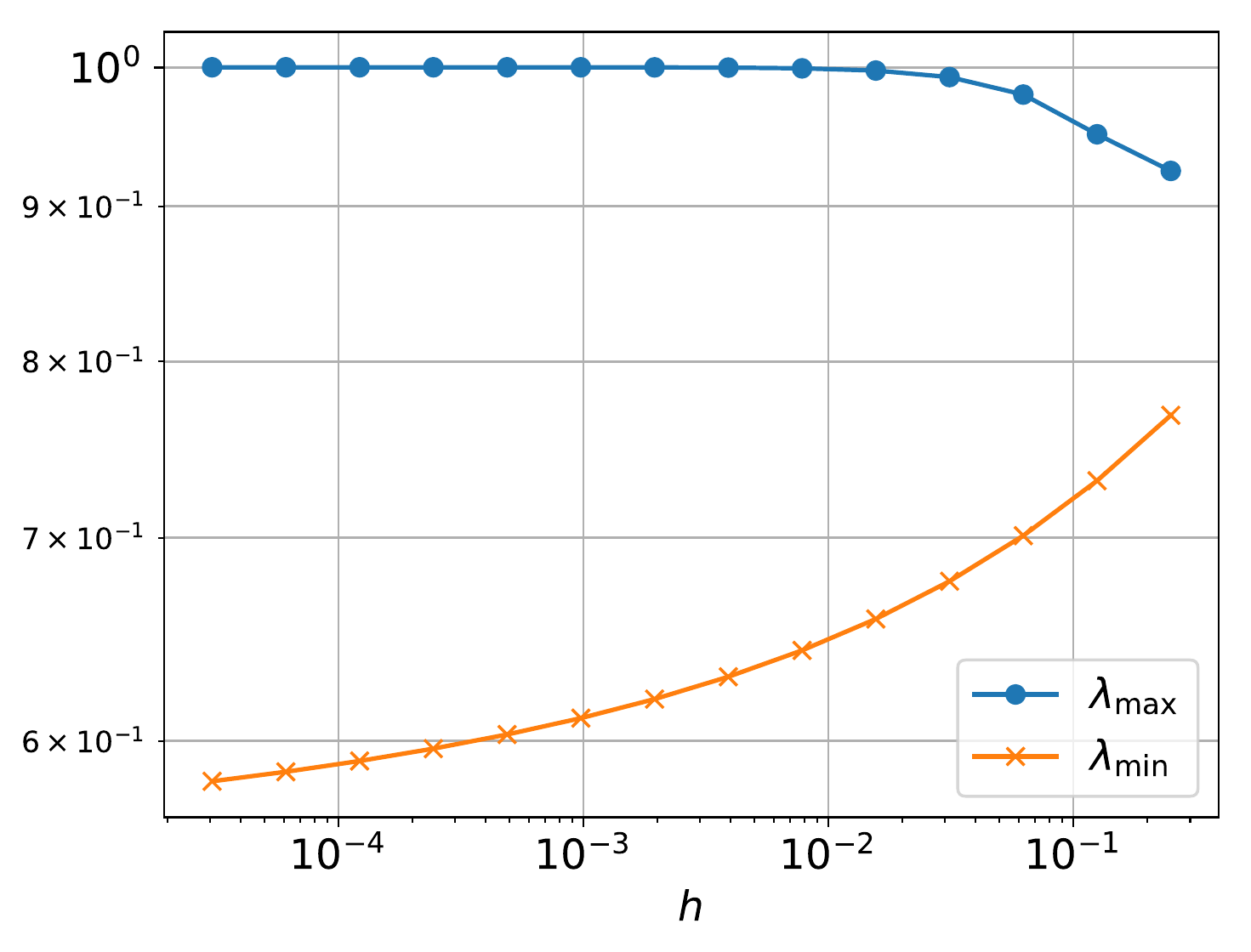}
    \caption{Behaviour of the generalized eigenvalues \(\lambda_{\min}\),
    and \(\lambda_{\max}\), see Example~\ref{ex:exist_Hs0}}.
    \label{fig:4}
  \end{figure}

  From~\eqref{eq:FourierRiesz} and polarization identity it follows that
  for \(u,v \in C^\infty_c(\O)\) we have the equality
  \(a(u,v) = (u,v)_A = \tfrac{1}{2}c(n,s)[u,v]_{H^{s}(\R)} =
  ((-\Delta)^{s}u,v)_{L^2(\R)}\).
  Therefore, Eringen problem~\eqref{wf} in this case reduces to
  that of solving the fractional Laplace problem in a bounded domain.
  For simple cases, it is possible to compute analytical solutions,
  see~\cite{di2012comparison,d2013fractional}.
  Namely, in our situation if \(f = 1\) then the analytical solution is
  \begin{equation*}
    u(x) = \frac{2^{-2s}\Gamma(n/2)}{\Gamma((n+2s)/2)\Gamma(1+s)} x^{s}(1-x)^{s}.
  \end{equation*}

  \begin{figure}
    \centering
    \begin{tabular}{cc}
      \includegraphics[width=0.43\columnwidth]{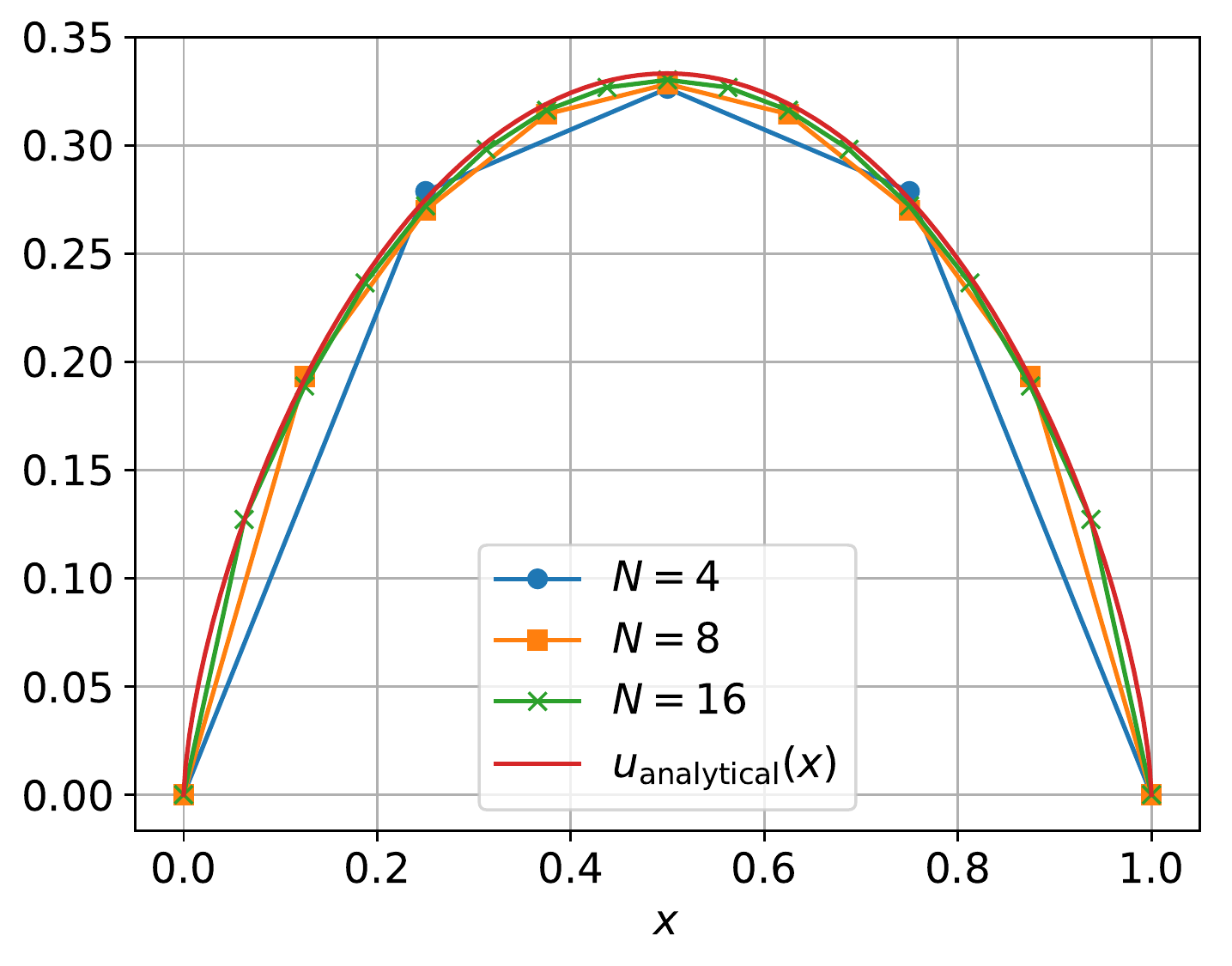} &
      \includegraphics[width=0.45\columnwidth]{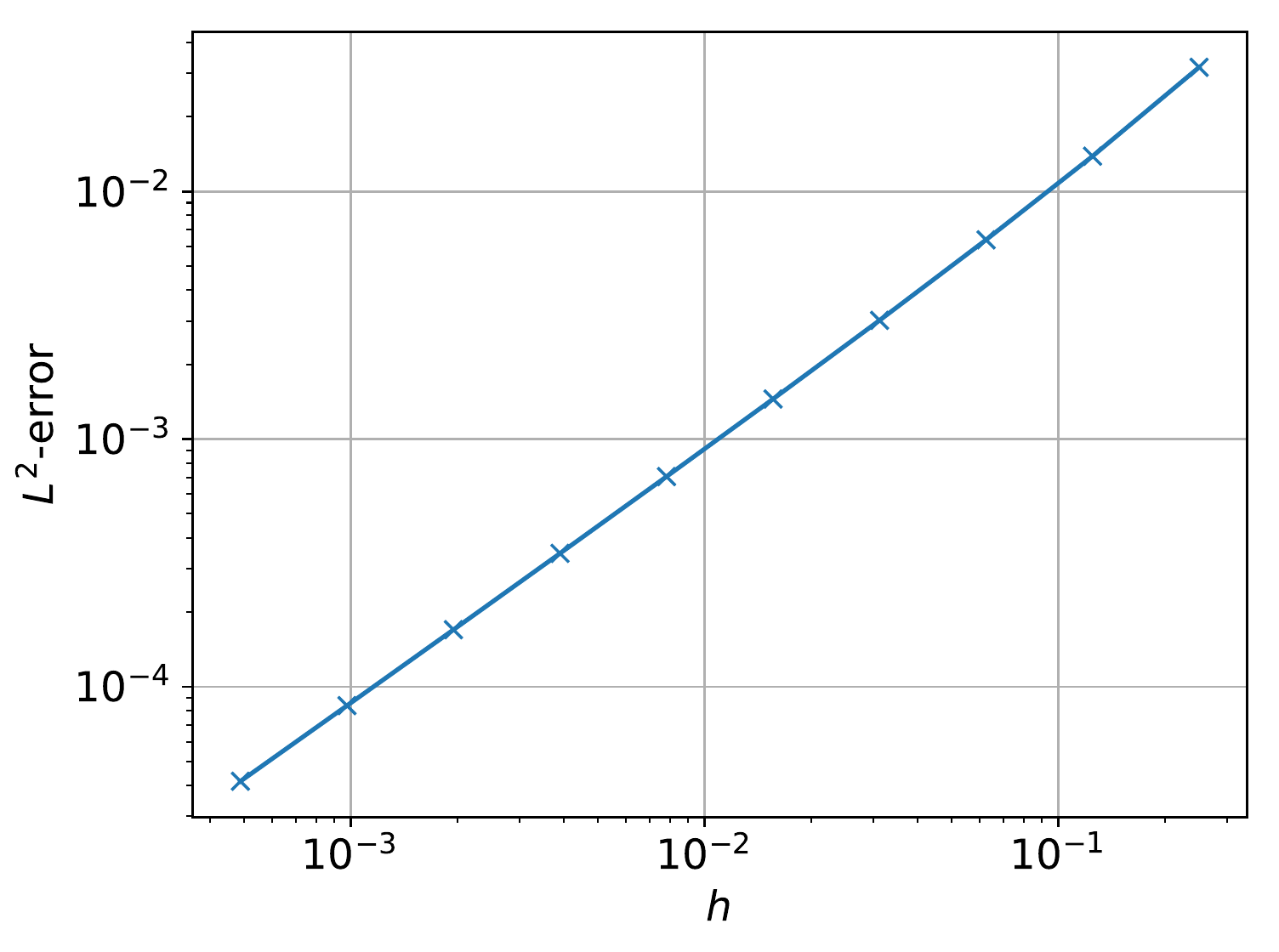}\\
      (a) & (b)
    \end{tabular}
    \caption{%
    (a): Numerical and analytical solutions to~\eqref{wf}
         with Riesz potential (hence also fractional Laplace problem)
         on some small mesh sizes, see Example~\ref{ex:exist_Hs0};
    (b): \(L^2(\O)\)-error between the finite element and
         analytical solution to~\eqref{wf} with Riesz potential kernel,
         see Example~\ref{ex:exist_Hs0}%
    }
    \label{fig:5}
  \end{figure}

  Figure~\ref{fig:5}~(a) shows the analytical and the numerical solutions for a
  few mesh sizes.
  Figure~\ref{fig:5}~(b) shows that the convergence rate with respect to
  \(L^2(\O)\) norm between the two solution is approximately linear
  (we estimate it at \(h^{1.06}\)), which is not unreasonable given the
  very low regularity of the analytical solution.
\end{example}

\subsection{Local-nonlocal mixed model and existence of solutions in \(H^{1}_0(\O;\Rn)\)}
\label{subsec:nlmix}

As a possible remedy to the ill-posedness of the Eringen's model, a so-called
local-nonlocal mixture model proposed by Eringen~\cite{eringen} is sometimes
utilized in the literature.
In this model one postulates that the strain-stress relation is given by the law
\begin{equation}\label{ss}
\sigma(x) =mC \varepsilon(x)+(1-m)  \displaystyle\int_{ \O} A(x,x')C\varepsilon (x')\,dx',
\end{equation}
where \(0< m <1 \) is a given weighting fraction between the local and the
nonlocal constitutive relations (see~\cite{pofuspi04} and the references therein).
Because of this construction as a convex combination of the local  and non-local laws,
arguably this model cannot be classified as a genuinely nonlocal model.
Nevertheless, it appears to be of interest for certain applications, for example
those involving modelling of nanotubes (see~\cite{romano} and the references therein).
For the sake of completeness of our investigation on existence of solutions to
Eringen's integral models, in this section we give a general existence result for
the mixture model.
The proof of such a result is elementary and straightforward, although as far as
the authors are aware it has not appeared in the literature.
Following the notation of Section 2, the local-nonlocal mixture model is:
\begin{equation}\label{ne2}
\left\{\begin{alignedat}{-1}
  -\diver(\sigma)&=f, && \text{in \(\O\)},\\
  \varepsilon &=\frac{1}{2}[\nabla u+(\nabla u)^\mathrm{T}], && \text{in \(\O\)},\\
  \sigma &=  mC\varepsilon    +  (1-m)\int_{\O} A(x,x')C \varepsilon(x')\,\mathrm{d}x',
  && \text{in \(\O\)},\\
\sigma \cdot \hat{n}&=g, && \text{on \(\Gamma_N\)},\\
u&=\bar{u}, &&\text{on \(\Gamma_D\)},\\
\end{alignedat}\right.
\end{equation}

Weak formulation of it can be stated as follows: find \(u\in H^1(\O;\Rn)\) with \(u=\bar{u}\)
on \(\Gamma_D\), such that
\begin{equation}\label{wf2}
  b(u,v)=l(v), \quad \forall v\in V,
\end{equation}
where \(V=\{ v\in H^1(\O;\R^n)\;:\; v=0 \text{\ on\ }\Gamma_D\},\)
and
\begin{equation*}\begin{split}
b(u,v)& =m\int_\O C\varepsilon_u(x):\varepsilon_v(x)\,\mathrm{d}x  +(1-m) a(u,v),
\end{split}\end{equation*}
and \(a(\cdot,\cdot)\) and \(\ell(\cdot)\) are defined in~\eqref{eq:al}.

We can now state and prove a general existence result for this problem.

\begin{theorem}
  Assume that the interaction kernel \(A \in L^2(\O\times\O)\) is (not necessarily strictly) positive definite.
  Then for each \(f\in L^2(\Omega;\Rn)\), \(\bar{u}\in H^{1/2}(\Gamma_D;\Rn)\) and \(g\in L^2(\Gamma_N; \Rn)\),
  the problem~\eqref{wf2} admits a unique solution.
\end{theorem}
\begin{proof}
  The proof of this theorem is a straightforward application of Theorem~\ref{thm:LM} in
  the Hilbert space \(H^1(\O; \R^n)\).
  Coercivity of the bilinear form \(b(\cdot,\cdot)\) follows from coercivity of the first term
  (owing to the classical Korn's inequality) and non-negativity of the second term
  (owing to the positive definiteness of the kernel \(A\)).
  Boundedness of \(b(\cdot,\cdot)\) follows from a straightforward application of
  Cauchy--Schwartz inequality.
  Continuity of \(\ell\) is owing to the continuous inclusion of
  \(H^1(\O;\R^n)\) into \(L^2(\O;\R^n)\) and the continuity
  of the trace operator \(H^1(\O;\R^n) \mapsto L^2(\Gamma_N;\R^n)\)~\cite{brezis}.
\end{proof}


\section{Extension of Eringen's model to heterogeneous materials}
\label{sec:heterog}

Assume now that the material stiffness tensor may vary spatially, that is,
that it is a bounded and measurable function \(x\mapsto C(x)\) satisfying
the bounds~\eqref{eq:localcoercivity} uniformly in \(x\).
The unfortunate consequence of this assumption is the fact that
the stiffness tensor and the non-local integral operator with kernel
\(A(\cdot,\cdot)\) no longer commute with each other and consequently
the bilinear form \(a(\cdot,\cdot)\) defined by~\eqref{eq:al} is no longer
symmetric.
Furthermore, we can no longer rely on the strict positive definiteness of
the kernel to infer the coercivity of the bilinear form \(a(\cdot,\cdot)\).

Whereas the symmetry of the bilinear norm can be easily recovered by e.g.\
substituting \([C(x)+C(x')]/2\) in place of \(C(x)\) in~\eqref{eq:al}, the
coercivity of \(a(\cdot,\cdot)\) is a much more delicate question.
One could for example try to utilize the knowledge that both the stiffness tensor and
the averaging operators are linear, self-adjoint, and positive definite and
therefore we can define a square root of each of these operators.
These ideas naturally lead to possible definitions
\begin{equation*}
  a(u,v) = \int_\O\int_\O A(x,x') \left(C^{1/2}(x)\varepsilon_u(x)\right):
  \left(C^{1/2}(x')\varepsilon_v(x')\right)\,\mathrm{d}x\,\mathrm{d}x'.
\end{equation*}
or
\begin{equation*}
  a(u,v) = \sum_{k,m=1}^\infty \int_\O \varphi_k(x)\varphi_m(x)C(x)
  \bigg[\int_\O \varepsilon_u(x')\lambda_k^{1/2}\varphi_k(x')\,\mathrm{d}x'\bigg]:
  \bigg[\int_\O \varepsilon_v(x'')\lambda_m^{1/2}\varphi_m(x'')\,\mathrm{d}x''\bigg]\,\mathrm{d}x,
\end{equation*}
where \(\lambda_k,\varphi_k\) are the eigenvalues and eigenfunctions of the
integral operator with the strictly positive definite kernel \(A\in L^2(\O\times\O))\)
provided by Theorem~\ref{thm:spectral}.
Note that while both of these definitions agree with~\eqref{eq:al} when
the stiffness tensor \(C\) is constant,
which, if any, of these mathematical constructions provides a useful model
of non-local elastic heterogeneous materials has to be assessed through a
rigorous model validation process, something which goes well beyond the scope
of this work or the expertise of the authors.
Nevertheless, one issue with these formulations is that there is no clear way
to get rid of the spatial dependent tensor \(C(x)\) in order to show coercivity
and boundedness of the bilinear form \(a(\cdot,\cdot)\), as in the proofs of
Propositions~\ref{prop:VAbound} and~\ref{prop:VAcoercive}.

With this disclaimer, we propose a different explicit model, which relies upon
algebraic properties of Riesz potentials.
Let us begin with the definition \(a(\cdot,\cdot)\) in~\eqref{eq:al}
and assume that \(C\) is constant.
Then for any \(u,v\in C_c^\infty(\O,\Rn)\) we can write:
\begin{equation}\label{eq:heterog_same_homog}
  \begin{aligned}
    a(u,v)
    &=
    \int_{\O}\int_{\O} \tilde{A}_{\alpha}(|x-x'|)
    C \varepsilon_u(x):\varepsilon_v(x')
    \,\mathrm{d}x'\,\mathrm{d}x
    \\&=
    \int_{\Rn} \int_{\Rn} \tilde{A}_{\alpha/2}(|x-x''|)
    \int_{\Rn} \tilde{A}_{\alpha/2}(|x''-x'|)
    C \varepsilon_u(x):\varepsilon_v(x')
    \,\mathrm{d}x'\,\mathrm{d}x''\,\mathrm{d}x
    \\&=
    \int_{\Rn} C
    \underbrace{\bigg[\int_{\O} \tilde{A}_{\alpha/2}(|x''-x|)\varepsilon_u(x)\,\mathrm{d}x\bigg]%
    }_{\varepsilon_u^{\text{nl}}(x'')}:
    \underbrace{\bigg[\int_{\O} \tilde{A}_{\alpha/2}(|x''-x'|)\varepsilon_v(x')\,\mathrm{d}x'\bigg]%
    }_{\varepsilon_v^{\text{nl}}(x'')}
    \,\mathrm{d}x'',
  \end{aligned}
\end{equation}
where we have used the semigroup property of the Riesz kernels~\cite[pp.~118]{stein}
to get to the second line from the first.
Note that the terms in squared brackets can be thought of as non-local
(averaged) strains, which are acted upon by the local stiffness tensor \(C\).
The non-local strains can be non-zero even outside of \(\Omega\)
thereby necessitating the integration over \(\Rn\) in the last term.
Perhaps the best way of thinking about this formula is that we consider
deformations of an infinite non-locally elastic body with the stiffness tensor
\(C\), while restricting the displacements to be zero outside of \(\Omega\).

The main reason for the derivation~\eqref{eq:heterog_same_homog} is that
the last term can be used as a new definition of the bilinear form
\(a(\cdot,\cdot)\) which remains symmetric even for spatially varying
material tensors \(C\):
\begin{equation}\label{eq:anew}
  a(u,v) = \int_{\Rn} C(x'')
  \bigg[\int_\O \tilde{A}_{\alpha/2}(|x''-x|)\varepsilon_u(x)\,\mathrm{d}x\bigg]:
  \bigg[\int_\O \tilde{A}_{\alpha/2}(|x''-x'|)\varepsilon_v(x')\,\mathrm{d}x'\bigg]
  \,\mathrm{d}x'',
\end{equation}
for \(u,v\in C^\infty_c(\O,\Rn)\) and \(\alpha\) satisfying the assumptions of
Proposition~\ref{prop:isomorph_Hs0}.
The questionable modelling part in the definition above is that if \(C\)
is only defined over \(\O\), it has to be arbitrarily extended onto \(\Rn\)
in such a way that the extension continues to be measurable and the
bounds~\eqref{eq:localcoercivity} continue to hold; for example one may
put \(C(x) = \underline{C}I\) for \(x\not\in\O\), where \(I\) is the identity
tensor.
On the bright side, with these definitions we easily generalize the desirable
mathematical properties established in the previous section to spatially
varying stiffness tensors.

\begin{proposition}\label{prop:a_bnd_coerc}
  Consider \(\alpha\) satisfying the assumptions of
  Theorem~\ref{thm:exist_Hs0}.
  The form \(a(\cdot,\cdot)\) defined by~\eqref{eq:anew} is bounded and coercive on
  \(H^s_0(\O;\Rn)\) with \(s = 1-\alpha/2\).
\end{proposition}
\begin{proof}
  Owing to the density of \(C^\infty_c(\O;\Rn)\) in \(H^s_0(\O;\Rn)\) it is sufficient
  to prove the claim for smooth functions with compact support.
  For an arbitrary \(u \in C^\infty_c(\O;\Rn)\) we have the string of inequalities:
  \begin{equation*}
    \begin{aligned}
      a(u,u)
      &\geq
      \underline{C}\int_\Rn
      \bigg[\int_\O \tilde{A}_{\alpha/2}(|x''-x|)\varepsilon_u(x)\,\mathrm{d}x\bigg]:
      \bigg[\int_\O \tilde{A}_{\alpha/2}(|x''-x'|)\varepsilon_u(x')\,\mathrm{d}x'\bigg]
      \,\mathrm{d}x''
      \\&=
      \underline{C}\int_\O\int_\O
      \tilde{A}_{\alpha}(|x-x'|)\varepsilon_u(x):\varepsilon_u(x')
      \,\mathrm{d}x\,\mathrm{d}x'
      \geq
      \frac{1}{2}\underline{C}(u,u)_A,
    \end{aligned}
  \end{equation*}
  where last line is derived from the first exactly as in~\eqref{eq:heterog_same_homog},
  and the two inequalities are owing to~\eqref{eq:localcoercivity}
  and Theorem~\ref{thm:L1_bnd_coerc}.
  Consequently, \(a(\cdot,\cdot)\) defines an inner product on \(C^\infty_c(\O;\Rn)\),
  and therefore as in the proof of Proposition~\ref{prop:VAbound} it is sufficient
  to bound \(a(u,u)\) in terms of \((u,u)_A\).
  Such a bound is established in the following string of inequalities:
  \begin{equation*}
    \begin{aligned}
      a(u,u)
      &\leq
      \overline{C}\int_\Rn
      \bigg[\int_\O \tilde{A}_{\alpha/2}(|x''-x|)\varepsilon_u(x)\,\mathrm{d}x\bigg]:
      \bigg[\int_\O \tilde{A}_{\alpha/2}(|x''-x'|)\varepsilon_u(x')\,\mathrm{d}x'\bigg]
      \,\mathrm{d}x''
      \\
      &\leq
      \overline{C}\int_\Rn
      \bigg[\int_\O \tilde{A}_{\alpha/2}(|x''-x|)\nabla u(x)\,\mathrm{d}x\bigg]:
      \bigg[\int_\O \tilde{A}_{\alpha/2}(|x''-x'|)\nabla u(x')\,\mathrm{d}x'\bigg]
      \,\mathrm{d}x''
      \\&=
      \overline{C}\int_\O\int_\O
      \tilde{A}_{\alpha}(|x-x'|)\nabla u(x):\nabla u(x')
      \,\mathrm{d}x\,\mathrm{d}x'
      =
      \overline{C}(u,u)_A,
    \end{aligned}
  \end{equation*}
  where in addition to the previously used arguments we utilize the orthogonality
  between symmetric and skew-symmetric second order tensors.
  Finally, the claim follows from Proposition~\ref{prop:isomorph_Hs0}.
\end{proof}

Therefore, we have all the necessary ingredients for applying
Theorem~\ref{thm:LM}, and are in position to state the following existence and
uniqueness result, which in view of~\eqref{eq:heterog_same_homog} generalizes
Theorem~\ref{thm:exist_Hs0}.

\begin{theorem}\label{thm:exist_Hs0_heterog}
    For \(\alpha\) and \(s\) as in Theorem~\ref{thm:exist_Hs0},
    the variational problem~\eqref{wf} with \(a(\cdot,\cdot)\) defined
    by~\eqref{eq:anew} and homogeneous Dirichlet boundary conditions
    (that is, \(\Gamma_D=\partial\O\) and \(\bar u=0\)) admits a unique solution
    in \(H^s_0(\O;\Rn)\).
\end{theorem}

We would like to remark that the model we propose aims to keep an explicit expression
of the integral kernel, as we find this interesting from the point of view of numerics
and mechanical applications.
Yet another option for extending the model to the heterogeneous case is to
take the \(s\)-power, in the functional analytic sense, of the heterogeneous local problem
operator as has been done for the scalar fractional elliptic equation definition in~\cite{stinga}.
In this case the bilinear form admits a representation as a nonlocal integral,
whose kernel can be estimated, plus a local bilinear form, the corrector term.

Another interesting question, which we cannot answer presently, is whether one
can reduce the domain of integration for the outer integral (that is, the integral
over non-local strains with respect to \(\mathrm{d}x''\)) in~\eqref{eq:anew}
to \(\O\) and
still maintain coercivity in a suitable function space, such as \(H^s_0(\Omega;\Rn)\)?
(Boundedness of the bilinear form obtained in this fashion with respect to
\(H^s(\Omega;\Rn)\) norm, \(s=1-\alpha/2\) is quite straightforward.)
Intuitively a positive answer to this question seems plausible, as we still include
the singularities of the kernel into the integration domain.
Additionally, we have repeated the numerical experiment, which resulted in Figure~\ref{fig:4}
but with such a modification of the bilinear form~\eqref{eq:anew}.
The results are shown in Figure~\ref{fig:7}, which qualitatively agree very well
with those shown in Figure~\ref{fig:4}.

\begin{figure}
  \centering
  \includegraphics[width=0.45\columnwidth]{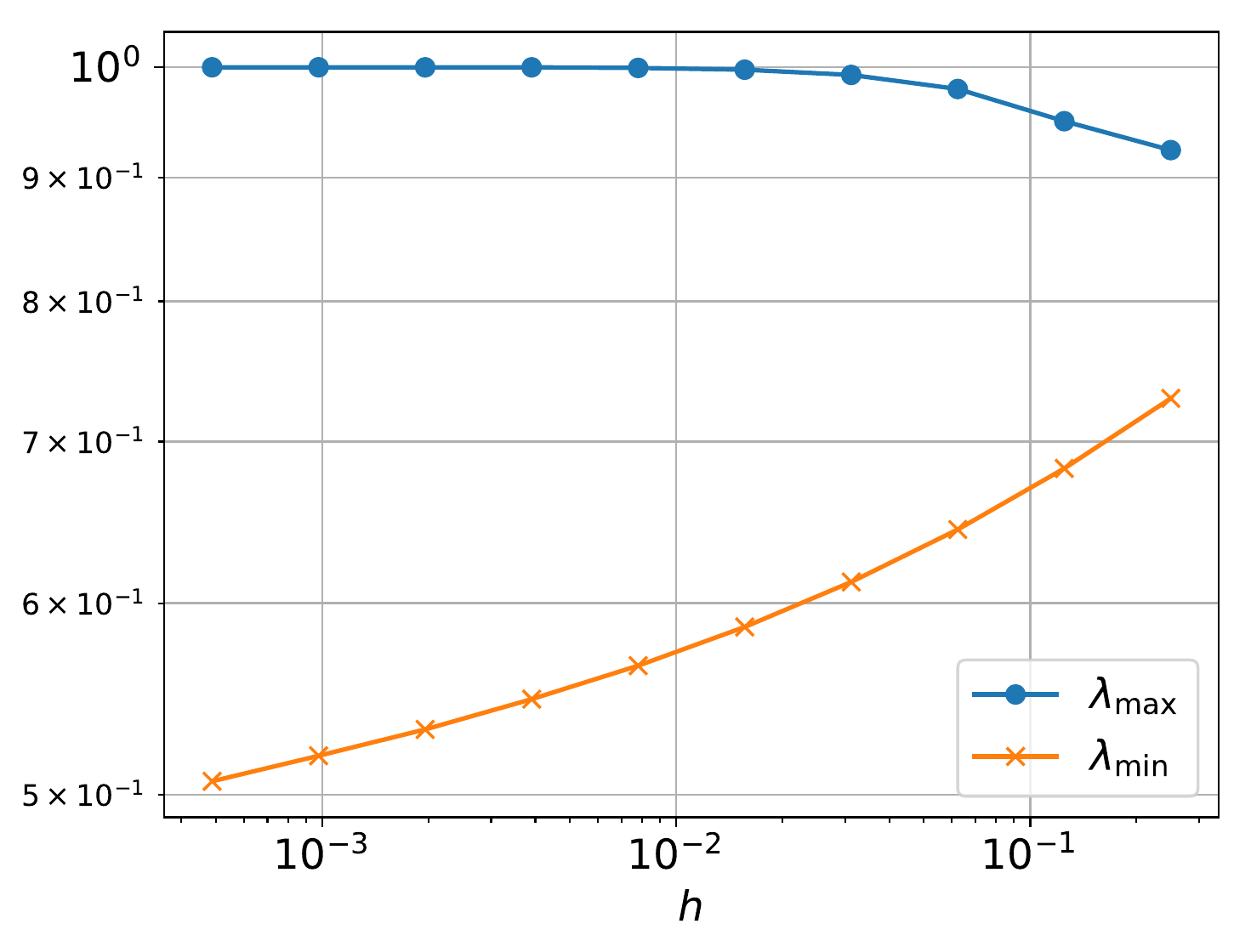}
  \caption{Behaviour of the generalized eigenvalues \(\lambda_{\min}\),
  and \(\lambda_{\max}\) numerically characterizing the coercivity and
  boundedness of the modification of the bilinear form~\eqref{fig:7}
  with respect to \(H^s(\O;\Rn)\)-norm.
  The case of \(\alpha=s=2/3\) is shown.}
  \label{fig:7}
\end{figure}

\section*{Acknowledgements}
A.E.'s research at DTU is funded by the Villum Fonden through
the Villum Investigator Project InnoTop.
The work of J.C.B. is funded by FEDER EU and Ministerio de Econom\'ia y
Competitividad (Spain) through grant MTM2017-83740-P.
J.C.B. acknowledges an interesting discussion on the subject of this paper with
Carlos Mora-Corral.

\bibliography{EringenExistence}

\bigskip
{\small\noindent
Anton Evgrafov: Department of Mathematical Sciences, \\
Norwegian University of Science and Technology (NTNU), N--7491 Trondheim,
Norway \\ 
and\\
Department of Mechanical Engineering, Technical University of Denmark,\\
Akademivej, Building 358, DK--2800 Kgs.\ Lyngby, Denmark. \\
{\tt aaev@mek.dtu.dk}

\bigskip\noindent
Jos\'e C. Bellido: E.T.S.I. Industriales, Department of Mathematics,\\
University of Castilla-La Mancha, 13.071-Ciudad Real, Spain.\\
{\tt josecarlos.bellido@uclm.es}
}

\end{document}